\RequirePackage{fix-cm}
\documentclass[10pt]{article}
\usepackage[utf8]{inputenc}
\usepackage{tikz-cd}
\usetikzlibrary{cd}
\usepackage{amsmath,amssymb,amsthm,amsfonts,amscd, url,mathtools}
\usepackage{mathrsfs}
\usepackage{tikz}
\usetikzlibrary{calc}
\usepackage{zref-savepos}
\usepackage{tikz}
\usepackage[title]{appendix}
\usepackage{graphicx,wrapfig}
\usepackage{float}
\graphicspath{ {images/} } 
\usepackage[english]{babel}
\usepackage[autostyle]{csquotes}
\usepackage{hyperref}
\usepackage{geometry}
\usepackage{cite}

\DeclareMathSizes{10}{10}{6}{4}

\DeclareFontFamily{U}{MnSymbolC}{}
\DeclareSymbolFont{MnSyC}{U}{MnSymbolC}{m}{n}
\DeclareFontShape{U}{MnSymbolC}{m}{n}{
	<-6>  MnSymbolC5
	<6-7>  MnSymbolC6
	<7-8>  MnSymbolC7
	<8-9>  MnSymbolC8
	<9-10> MnSymbolC9
	<10-12> MnSymbolC10
	<12->   MnSymbolC12}{}
\DeclareMathSymbol{\intprod}{\mathbin}{MnSyC}{'270}

\newtheoremstyle{problemstyle}
{\topsep} % Space above
{20pt} % Space below
{} % Body font
{} % Indent amount
{\scshape} % Theorem head font
{\newline} % Punctuatnion after theorem head
{1em} % Space after theorem head
{} % Theorem head spec (can be left empty, meaning `normal')

\newtheorem*{conjecture}{Conjecture}
\newtheorem{prop}{Proposition}
\numberwithin{prop}{section}
\newtheorem{cor}{Corollary}
\numberwithin{cor}{section}
\newtheorem{lem}{Lemma}
\numberwithin{lem}{section}

\newtheorem{thm}{Theorem}
\numberwithin{thm}{section}

\theoremstyle{problemstyle}

\theoremstyle{definition}
\newtheorem{defs}{Definition}
\numberwithin{defs}{section}
\newtheorem{egs}{Example}
\numberwithin{egs}{section}

\theoremstyle{remark}
\newtheorem{rem}{Remark}
\numberwithin{rem}{section}

\newcommand{\GL}{\text{GL}}

\newcommand{\id}{\text{Id}}

\newcommand{\into}{\hookrightarrow}

\newcommand{\bR}{\mathbb{R}}

\newcommand{\msO}{\mathscr{O}}

\newcommand{\mfX}{\mathfrak{X}}

\newcommand{\Hom}{\text{Hom}}
\newcommand{\End}{\text{End}}

\newcommand{\bge}{\begin{equation*}}
	\newcommand{\ene}{\end{equation*}}

\newcommand{\bra}{\langle}
\newcommand{\ket}{\rangle}

\title{Stratified Vector Bundles: Examples and Constructions}
\author{Ethan Ross}

\begin{document}
	
	\maketitle
	
\section{Introduction}
A stratified space can roughly be thought of as topological space which come equipped with a partition into smooth manifolds.This class of space arises in a wide variety of contexts like the study of real analytic varieties \cite{whitney_local_2015}, function spaces \cite{mather_stratifications_1973} \cite{thom_ensembles_1969}, orbispaces \cite{crainic_orbispaces_2018}, symplectic reduction \cite{sjamaar_stratified_1991}, and intersection homology theory \cite{goresky_intersection_1980}. As with manifolds, there is a way of equipping stratified spaces with smooth structures and obtaining stratified versions of all the usual flora and fauna of differential topology, like smooth maps, de Rham complexes, integration, etc. However, the notion of the tangent bundle of a stratified space \cite{pflaum_analytic_2001} is not actually a vector bundle since the rank of each tangent space may vary. Thus, we develop in this paper a notion of a stratified vector bundle to fit stratified tangent bundles into a wider context.

Another major motivation for developing stratified vector bundles is to use them for the purposes of quantization. In particular, in the Kostant-Souriau-Weil picture of quantization, one requires three pieces of initial data: a symplectic manifold, a complex line bundle with connection, and a polarization\footnote{A polarization of a symplectic manifold is a Lagrangian involutive subbundle of the complexified tangent bundle.}; all subject to various compatibility conditions \cite{hall_quantum_2013}. In particular, we have 3 vector bundles in this picture: the complexified tangent bundle, the line bundle, and the polarization. Quantization should morally produce a Hilbert space, so various auxiliary constructions have been devised such as half-forms and metaplectic corrections to do this. However, “singular” versions of polarizations and symplectic manifolds have appeared, generally through symplectic reduction \cite{sjamaar_stratified_1991} or integrable systems \cite{hamilton_locally_2010}. In these cases,  we lose at least one of the vector bundles required for the constructions used to construct Hilbert spaces. One possible solution for these problems, is to slightly widen the category of vector bundles to a setting more suitable for singular, usually stratified objects. Thus, opening the door for stratified versions of half-forms or metaplectic corrections.

As will be discussed in section 2, this paper is not the first one to define objects called “stratified vector bundles”. The authors Kucharz-Kurdyka \cite{kucharz_stratified-algebraic_2018} and Baues-Ferrario \cite{baues_k-theory_2003} each have inequivalent definitions of stratified vector bundles that arise from their respective work in algebraic geometry and K-theory. The Kucharz-Kurdyka definition is a special case of the definition given in this paper and we believe that the Baues-Ferrario one is as well. This paper is geared more towards a differential geometric approach than either of the previous two mentioned works, so there is no conflict in the material.  It has also come to the attention of the author that Scarlett in \cite{scarlett_smooth_2023} has also developed, independently, a definition of smooth stratified vector bundles. Scarlett's definition is a fair bit stronger than the definition provided here and is used to produce objects not considered in this paper like generalizations of frame bundles. Thus, our two papers are complementary.

It should also be noted that many of the examples considered in this paper could fit into an even wider context of stratified groupoids as developed by Farsi-Pflaum-Seaton \cite{farsi_differentiable_2023}. The differentiable stratified algebroids defined in that paper  can be be characterized as differentiable stratified vector bundles as defined in section 2 with a suitable bracket on sections. It would be worthwhile to study if any interesting  representation theory could be obtained from stratified groupoids and stratified vector bundles.

In section 2, we will discuss the basic definitions of the theory of stratified vector bundles. First, we  give the definition of stratified spaces that will be used in this paper. As we will discuss, there are a few different definitions in the literature so it is necessary to fix one for our purposes. We will also discuss Pflaum's definition of a smooth structures \cite{pflaum_analytic_2001} on stratified spaces since these are where the first non-trivial examples of stratified vector bundles arise. From there we will give the definition of a stratified vector bundle and stratified versions of vector bundle morphisms and sections. To close, we give an alternative characterization of stratified vector bundles in terms of their scalar multiplication in the spirit of Grabowski-Rotkiewicz \cite{grabowski_higher_2009}.

In section 3, we then move on to two important families of examples of stratified vector bundles. The first family arises from the theory of Stefan-Sussmann singular foliations. These are a natural and well-studied class of \enquote{singular vector bundles} arising from Lie groupoids and the theory of Riemannian foliations. We show that if the leaves can be grouped together into embedded submanifolds in a suitable fashion, then there is a stratified vector bundle underlying the foliation. The next family comes from equivariant vector bundles. Here, we will develop a procedure for taking a modified quotient of the total space of a vector bundle in order to obtain a stratified vector bundle on the quotient of the base. This procedure generalizes the construction of vector bundles over the base of a principal bundle by taking quotients of equivariant vector bundles over the total space \cite{vistoli_grothendieck_2005}. 

In the final section, we extend results about applying linear functors fibre-wise to vector bundles to the case of stratified vector bundles. Here we will introduce a version of \enquote{local triviality} for stratified vector bundles, called an injective structure, and a variant of the famous Whitney A condition for such bundles. We then prove that provided the linear functor preserves orthogonal projections, then it can be fibre-wise applied to a Whitney A stratified vector bundle to canonically obtain another stratified vector bundle.

In a future paper, we will extend the results in this paper about equivariant vector bundles to VB groupoids \cite{bursztyn_vector_2016}. We also think the technology of open quotient vector bundles and kernel maps developed by Resende-Santos \cite{resende_open_2017} could be useful for applying more general linear functors to more general stratified vector bundles, not just Whitney A ones. Furthermore, other elements of vector bundle theory like characteristic classes should be extendable to stratified vector bundles, at least in the case of stratified vector bundles over quotient spaces.

I would like to thank the following people for their contributions to this paper. My supervisor Dr. Lisa Jeffrey for pointing me towards resources and helping with the massive job of editing the drafts of this paper. I would also like to thank Dr. Maarten Mol for all the helpful discussions about the minutiae of stratified spaces, particularly the annoying point set topology issues that arise in this singular world. Finally, I would like to thank the reviewer for pointing out numerous errors and typos, as well as substantial improvements to the bibliography.

\tableofcontents

\section{Stratified Vector Bundles}
Very roughly speaking, a stratified space is just a topological space $X$ with a partition $\Sigma$ into smooth manifolds. A stratified vector bundle can then be understood as a special kind of stratified space where the pieces of the partition are demanded to not just be any arbitrary collection of smooth manifolds, but must be the total spaces of smooth vector bundles. In this section, we will be making all the above ideas precise, giving elementary examples, and we will finish with an alternate characterization in the spirit of Grabowski-Rotkiewicz \cite{grabowski_higher_2009} in terms of monoid actions. This should provide at least a partial justification that stratified vector bundles are the \enquote{correct} notion of a vector bundle over a singular space. 

\subsection{Stratified Spaces}
As we will discuss, there are many competing definitions of stratified spaces in the literature. So to begin, we shall give the definition to be used in this paper. 

\begin{defs}\label{strdef}
	A {\bf stratified space} is a pair $(X,\Sigma)$, where $X$ is a Hausdorff, second countable, paracompact topological space and $\Sigma$ is a locally finite partition of $X$ into connected, locally closed subspaces satisfying the following conditions.
	\begin{itemize}
		\item[(i)] Each piece $S\in\Sigma$ is a topological manifold in the subspace topology.
		\item[(ii)] Each piece $S\in\Sigma$ is equipped with the structure of a smooth manifold.
		\item[(iii)] ({\bf Frontier Condition}) If $S,R\in\Sigma$ are two pieces with $S\cap\overline{R}\neq\emptyset$, then $S\subset \overline{R}$.
	\end{itemize}
	If $(X,\Sigma)$ is a stratified space, call $\Sigma$ a {\bf stratification} of $X$.
\end{defs}

\begin{rem}
	There are a few different definitions of a stratified space in the literature. The definition being used in this paper essentially goes back to Mather \cite[Definition 8.1]{mather_notes_2012}. There is also the classical inductive definition used by Goresky and MacPherson \cite[Definition 1.1]{goresky_intersection_1980} where we have a filtration of a topological space $X$
	\begin{equation}\label{filt}
		\emptyset=X_{-1}\subset X_0\subset X_1\subset\cdots\subset X_n=X,
	\end{equation}
	so that the connected components of each of the differences $X_i\setminus X_{i-1}$ are smooth manifolds of dimension $i$ and so that locally there exists a homeomorphism
	\begin{equation}\label{loccone}
		X_i\cong \bR^i\times C(Y),
	\end{equation}
	where $C(Y)$ is the cone of a stratified space of lower dimension. This is a special case of Definition \ref{strdef} since given a stratified space in the sense of \ref{strdef} $(X,\Sigma)$, we can form a filtration of $X$ by defining for each $i\geq 0$
	\bge
	\Sigma_i:=\{S\in\Sigma \ | \ \dim S\leq i\}
	\ene
	and setting
	\bge
	X_i:=\bigcup_{S\in\Sigma_i}S.
	\ene
	If there exists $n$ so that $X_n=X$, then we get a filtration like equation (\ref{filt}) where the connected components of the complements $X_i\setminus X_{i-1}$ are smooth manifolds of dimension $i$. However, there is no guarantee that equation (\ref{loccone}) will hold. 
	
	Nevertheless, Mather \cite[Theorem 8.3]{mather_stratifications_1973} was able to show that if $X$ comes equipped with a Whitney B structure (see Definition \ref{whitdef}), then (\ref{loccone}) will hold. Furthermore, Nocera-Volpe \cite[Theorem 3.8]{nocera_whitney_2023} showed that the homeomorphisms can be chosen in such a way that they are conically smooth in the sense of Ayala-Francis-Tanaka \cite{ayala_local_2017}.
	
	Another definition that looks a lot more like the one given above is Pflaum's notion of a decomposed space \cite[Definition 1.1.1]{pflaum_analytic_2001}.  A stratification in this context is then a structure associated to germs of closed sets. This is what Crainic-Mestre \cite{crainic_orbispaces_2018} refer to as a  \enquote{germ-stratification}. Thankfully, every decomposed space induces a germ-stratification and hence the concepts are coherent. It should also be noted that this idea of using germs of closed sets to define a stratification goes back to Mather \cite[pages 199-200]{mather_stratifications_1973}. Although his definition is even more general as his decomposed spaces (what he terms prestratified spaces) need not satisfy axioms (i) and (ii) in Definition \ref{strdef}, that is, the pieces of the partition need not be topological manifolds.
	
	Another notion of stratification which only slightly differs from the one used in this paper can be found in Crainic-Mestre \cite{crainic_orbispaces_2018} where the frontier condition is slightly different. Not only do the authors assume that if two strata $S,R$ satisfy $S\cap \overline{R}\neq\emptyset$, then $S\subset\overline{R}$; but furthermore, they also assume that if $S\neq R$, then $\dim(S)<\dim(R)$. This may be too powerful an assumption for some of the more interesting examples we will be discussing later on.
\end{rem}

\begin{egs}
	Let $M$ be a smooth manifold. Letting $\Sigma$ denote the set of connected-components of $M$, then $(M,\Sigma)$ is a stratified space.
\end{egs}

\begin{egs}
	Let $X$ be a finite CW complex and let $\Sigma$ denote the collection of relative interiors of disks attached to $X$. Since each $D\in\Sigma$ of dimension $n$ can be identified with the unit open ball centred at the origin in $\bR^n$, we can equip $D$ with a smooth structure. The frontier condition follows from the attaching maps.
\end{egs}

\begin{egs}
	Suppose $M$ is a smooth manifold and $G$ a Lie group acting on $M$ properly, that is, the map
	\bge
	G\times M\to M\times M;\quad (g,x)\mapsto (g\cdot x,x)
	\ene
	is proper. For the following discussion, we will be following Pflaum \cite{pflaum_analytic_2001}.
	
	For each $x\in M$, write $G_x$ for the stabilizer group of $x$. Then, for subgroup $H\leq G$, we may define the subsets
	\begin{align}\label{stabsubs}
		M_H&:=\{x\in M \ | \ G_x=H\}\\
		M_{(H)}&:=\{x\in M \ | \ G_x\text{ is conjugate to } H\}\\
		M^H&:=\{x\in M \ | \ H\subset G_x\}
	\end{align}
	The set $M_{(H)}$ is called {\bf the set of points with orbit type $H$.}
	
	Note that if $x\in M_H$, then $H$ is compact. Furthermore, writing $N_x:=T_x M/ T_x(G\cdot x)$ for the normal space to the orbit of $x$, we get a $G$-equivariant diffeomorphism from $G\times_H N_x$ onto a $G$-invariant open neighbourhood $U$ of $x$ such that $[g,0]$ gets mapped to $g\cdot x$. Such a map is called a {\bf slice chart about $x$}. 
	
	Observe that in a slice, we have
	\bge
	M_H=N_G(H)\times_H N_x^H=N_G(H)/H\times N_x^H
	\ene
	and
	\bge
	M_{(H)}=G\times_H N_x^H=G/H\times N_x^H,
	\ene
	where $N_G(H)\leq G$ denotes the normalizer of $H$. Hence, it follows that the connected components of $M_H$ and $M_{(H)}$ are embedded submanifolds of $M$. Furthermore, using slices and induction on the dimension of $M$, one may also show as Pflaum does that
	\begin{itemize}
		\item There are only finitely many orbit types, i.e. the set $\{M_{(H)} \ | \ H\leq G\}$ is finite. 
		\item If $H,K\leq G$ are subgroups such that $H$ is conjugate to a subgroup of $K$, then $M_{(H)}\cap \overline{M_{(K)}}$ is open and closed in $M_{(H)}$.
	\end{itemize}
	Thus, if we partition $M$ by orbit types, then pass to connected components, the resulting partition is a stratification of $M$ by embedded locally closed submanifolds. This leads us to the next definition.
	
	\begin{defs}\label{orbitypestrat}
		Write $\mathcal{S}_G(M)$ for the stratification of $M$ by connected components of orbit type equivalence classes. Call $\mathcal{S}_G(M)$ the {\bf stratification of $M$ by orbit types.}
	\end{defs}
	
	This isn't the only stratification we get out of the action of $G$ on $M$. Since the action is proper and quotients by group actions are open, it follows that $M/G$ is Hausdorff. Furthermore, one observes that in a slice, we have
	\bge
	M_{(H)}/G\cong N_x^H
	\ene
	Furthermore, the elements of $\mathcal{S}_G(M)$ are $G$-invariant submanifolds. Hence, $M/G$ comes equipped with a canonical stratification with the strata being of the form $S/G$, where $S\in\mathcal{S}_G(M)$. 
	
	\begin{defs}\label{canonicalstrat}
		Define
		\bge
		\mathcal{S}_G(M/G):=\{S/G \ | \ S\in\mathcal{S}_G(M)\}.
		\ene
		Call $\mathcal{S}_G(M/G)$ the {\bf canonical stratification of $M/G$.}
	\end{defs}
\end{egs}

\begin{egs}
	Due to Pflaum-Posthuma-Tang \cite{pflaum_geometry_2014}, we can generalize the above to proper Lie groupoids. Suppose now that we have a Lie groupoid $G\rightrightarrows M$ \cite[Definition 13.2]{crainic_lectures_2021} with source and target maps $s$ and $t$, respectively. We say the groupoid is {\bf proper} if the map
	\bge
	G\to M\times M;\quad g\mapsto (s(g),t(g))
	\ene
	is proper. For each $x\in M$, define its isotropy group $G_x$ by
	\bge
	G_x:=s^{-1}(x)\cap t^{-1}(x).
	\ene
	$G_x$ is canonically a Lie group. Furthermore, we can define the orbit through $x$ by
	\bge
	\msO_x:=t(s^{-1}(x)).
	\ene
	Since $G$ is proper, it follows that $\msO_x$ is an embedded submanifold of $M$. Furthermore, observe that $G_x$ acts linearly on $T_x\msO_x$, and hence acts linearly on the normal space $N_x:=T_xM/T_x\msO_x$. $N_x$ is called the normal representation of $G$ at $x$. 
	
	Say two points $x,y\in M$ have the same {\bf Morita type} if there exists an isomorphism of Lie groups $\phi:G_x\to G_y$ and a linear isomorphism $\psi:N_x\to N_y$ such that the diagram commutes
	\bge
	\begin{tikzcd}
		G_x\times N_x \arrow[r,"\phi\times \psi"]\arrow[d]& G_y\times N_y\arrow[d]\\
		N_x\arrow[r,"\psi"] & N_y
	\end{tikzcd}
	\ene
	where the vertical maps are the action maps. This defines an equivalence relation on $M$, hence a partition. Passing to the connected components, we obtain a partition $\mathcal{S}_G(M)$ on $M$, called the {\bf stratification by Morita types}. It follows using the linearization theorem for proper Lie groupoids \cite[Theorem 1]{crainic_linearization_2013} that $\mathcal{S}_G(M)$ indeed defines a stratification of $M$ into locally closed submanifolds \cite[Theorem 5.3]{pflaum_geometry_2014}. Furthermore, if we do the special case of an action groupoid $G\times M\rightrightarrows M$ with the action of $G$ on $M$ being proper, then the Morita type stratification agrees with the orbit type stratification.
	
	Also just as above, the space of orbits $M/G$ is canonically stratified \cite[Corollary 5.4]{pflaum_geometry_2014}. Given an element $S\in \mathcal{S}_G(M)$, its image in $M/G$ is canonically a smooth manifold. Hence, we obtain the {\bf canonical stratification} of $M/G$, denoted $\mathcal{S}_G(M)$. 
\end{egs}

\begin{defs}\label{stratmorph}
	Let $(X_1,\Sigma_1)$ and $(X_2,\Sigma_2)$ be stratified spaces. A stratified morphism is a continuous map $f:X\to Y$ such that for all $S_1\in \Sigma_1$, there exists $S_2\in\Sigma_2$ such that
	\begin{itemize}
		\item[(i)] $f(S_1)\subset S_2$
		\item[(ii)] $f|_{S_1}:S_1\to S_2$ is smooth. 
	\end{itemize}
\end{defs}

\begin{egs}
	Let $G\rightrightarrows M$ be a proper Lie groupoid. Then, with respect to the Morita type stratification $\mathcal{S}_G(M)$ and the canonical stratification $\mathcal{S}_G(M/G)$ on $M/G$,  the quotient map 
	\bge
	\pi:M\to M/G
	\ene
	is a stratified morphism. 
\end{egs}

\subsubsection{Smooth Structures}
The first non-trivial examples of a smooth vector bundles are the tangent bundles of smooth manifolds. A similar story will be true for stratified vector bundles, but first we need to say what smooth means in the context of stratified spaces. As with smooth manifolds, this involves charts and atlases. However, unlike the usual notion of charts, the dimension of the co-domain is allowed to vary. Taken altogether, we will be equipping our stratified spaces with the structure of a compatible ``differentiable structure''. Please see Chapter 2 of the monograph \cite{navarro_gonzalez_c-differentiable_2003} for an extensive definition of the more general notion of a ``differentiable space''.

Much of the following discussion, including Definition \ref{smthdef}, are taken directly from Pflaum \cite{pflaum_analytic_2001}.

\begin{defs}\label{smthdef}
	Let $(X,\Sigma)$ be a stratified space. 
	\begin{itemize}
		\item[(1)] A {\bf chart} of $X$ \cite[page 26]{pflaum_analytic_2001} is an open subset $U\subset M$ and a topological embedding $\phi:U\to \bR^n$ as a locally closed subspace for some $n$ such that for all $S\in\Sigma$, $\phi(S\cap U)$ is an embedded submanifold of $\bR^n$.
		\item[(2)] Two charts $\phi:U\to \bR^n$ and $\psi:V\to \bR^n$ are said to be {\bf compatible} if for all $p\in U\cap V$, there exists open neighbourhood $W\subset W\cap V$ of $p$, open neighbourhoods $O_U,O_V\subset \bR^n$ of $\phi(W)$, $\phi(V)$ respectively, and a diffeomorphism $H:O_U\to O_V$ such that the diagram commutes
		\bge
		\begin{tikzcd}
			O_U\arrow[rr,"H"] && O_V\\
			\phi(W)\arrow[u] && \psi(W)\arrow[u]\\
			&W\arrow[ul,swap,"\phi"]\arrow[ur,"\psi"]&
		\end{tikzcd}
		\ene
		Now suppose we have two charts $\phi:U\to \bR^n$ and $\psi:V\to \bR^m$ where the dimension of the target could be different. Let $N=\text{max}\{n,m\}$, $\iota_j^N:\bR^j\into \bR^N$ be the inclusion of the first $j$ coordinates for $j\leq N$. Then $\phi$ and $\psi$ are {\bf compatible} if $\iota_n^N\circ \phi$ and $\iota_m^N\circ \psi$ are compatible. 
		\item[(3)] A {\bf smooth atlas} on $X$ consists of a covering of $X$ by compatible charts. Two atlases are equivalent if their union is again a smooth atlas. A {\bf smooth structure} on $X$ will be a choice of a maximal atlas. A stratified space together with a smooth structure will be called a {\bf differentiable stratified space}.
		\item[(4)] If $X$ is smooth, then a continuous map $f:X\to \bR$ is {\bf smooth} if for every chart $\phi:U\to \bR^n$, there exists smooth map $g:\bR^n\to \bR$ so that the diagram commutes
		\bge
		\begin{tikzcd}
			U\arrow[dr,swap,"f"]\arrow[r,"\phi"]& \bR^n\arrow[d,"g"]\\
			& \bR
		\end{tikzcd}
		\ene
		A map $F:X\to Y$ between differentiable stratified spaces is said to be {\bf smooth} if for all smooth $f:Y\to \bR$, $f\circ F:X\to \bR$ is smooth.
	\end{itemize}
\end{defs}

\begin{egs}
	As an elementary example, let us take $X=\bR$ endowed with the stratification
	\bge
	\Sigma=\{\bR_{<0},\{0\},\bR_{>0}\}.
	\ene
	Consider now the map
	\bge
	\phi:X\to \bR^2;\quad x\mapsto (x,|x|).
	\ene
	This is trivially a global chart, hence defines a smooth structure on $(X,\Sigma)$. One of the benefits of this structure is that now the absolute value function is smooth! Indeed, let
	\bge
	f:X\to \bR;\quad x\mapsto |x|
	\ene
	and
	\bge
	g:\bR^2\to \bR;\quad (x,y)\mapsto y.
	\ene
	Then, the diagram commutes
	\bge
	\begin{tikzcd}
		X\arrow[r,"\phi"]\arrow[dr,swap,"f"]& \bR^2\arrow[d,"g"]\\
		& \bR
	\end{tikzcd}
	\ene
	Hence, $f$ is smooth.
\end{egs}

\begin{egs}
	Let $M$ be a smooth manifold and let $\Sigma$ be a stratification of $M$ into embedded submanifolds. Suppose $\phi:U\subset M\to \bR^n$ is a chart of $M$ in the usual sense of a diffeomorphism onto an open subset of $\bR^n$. Then, $\phi:U\to \bR^n$ is also a chart in the stratified sense. Furthermore, an atlas $\mathcal{A}$ of $M$ in the usual sense is also an atlas of $(M,\Sigma)$ in the stratified sense as well.
\end{egs}

\begin{egs}\label{smoothorbits}
	Let $G$ be a connected Lie group and $M$ a smooth proper $G$-space. As we've seen,  $M$ together with its stratification into orbit types $\mathcal{S}_G(M)$ is canonically a differentiable stratified space since the pieces of $\mathcal{S}_G(M)$ are locally closed embedded submanifolds of $M$. The non-trivial fact is that the quotient space $M/G$ and its canonical stratification $\mathcal{S}_G(M/G)$ is also a differentiable stratified space. 
	
	The key to showing this is as follows. Let $H$ be a compact Lie group and $V$ an $H$-representation. Then, Hilbert's Finiteness Theorem for compact Lie groups \cite[Theorem 8.14.A]{weyl_classical_2016} states that the ring of $H$-invariant polynomials on $V$, $\bR[V]^H$, is finitely generated. Choosing a list of generators $p_1,\dots,p_k\in\bR[V]^H$, we get a map
	
	\begin{equation}\label{hilbertmap}
		p:V\to \bR^k;\quad v\mapsto (p_1(v),\dots,p_k(v)).
	\end{equation}
	
	Schwarz \cite{schwarz_smooth_1975} and Mather \cite{mather_differentiable_1977} both independently showed that $p$ is proper and that the induced map 
	
	\begin{equation}
		\widetilde{p}:V/H\to \bR^k
	\end{equation}
	is a proper topological embedding with image a semi-algebraic set. As a consequence of this fact, suppose we pick a point $x\in M$ with $G_x=H$ and let $U\cong G\times_H V$ be a slice chart about $x$, with $V$ an $H$-representation. Note that 
	\bge
	U/G\cong V/H
	\ene
	and $H$ is compact. Hence, choosing generators of $\bR[V]^H$, we get a singular chart $\widetilde{p}:U/G\into \bR^k$ for some $k$. The main ingredient for showing that any two charts constructed in this fashion are compatible is the following theorem of Schwarz.
	
	\begin{thm}[Schwarz {{\cite[Theorem 1]{schwarz_smooth_1975}}}]\label{schwarz}
		Let $H$ be a compact Lie group and $V$ and $H$ representation. With $p$ as in equation (\ref{hilbertmap}), the induced map
		\begin{equation}\label{hilb}
			p^*:C^\infty(\bR^k)\to C^\infty(V)^H
		\end{equation}
		is a surjection.
	\end{thm}
	
	Mather \cite[Theorem 1]{mather_differentiable_1977} was able to strengthen Schwarz's Theorem by showing that $p^*$ in Equation (\ref{hilb}) is a split surjection, that is, it admits a right inverse. Using this, one can show as Pflaum does \cite[Remark 4.4.4]{pflaum_analytic_2001}, that charts on $M/G$ arising from maps as in Equation (\ref{hilbertmap}) are compatible with one another. 
	
	One nice consequence of compatibility is that we can give an exact characterization of the smooth functions on $M/G$. Indeed, let $\pi:M\to M/G$ be the quotient map and $U\subset M/G$ an open subset. Then, $f:U\to \bR$ is smooth if, and only if, $f\circ \pi:\pi^{-1}(U)\to \bR$ is smooth. In particular, 
	\begin{equation}\label{quotmap}
		C^\infty(M/G)\cong C^\infty(M)^G.
	\end{equation}
	
\end{egs}

\begin{rem}
	In general, differentiable stratified spaces can still be quite wild in nature. To grapple with these objects, various extra conditions can be imposed on the smooth structure and the topology. Two of the most widely used are the Whitney conditions. In this paper, we will only be discussing the Whitney A condition as it is the most natural one to generalize to the stratified vector bundle context. As was shown by Scarlett, it is also possible to extend Whitney C \cite[Section 3.4]{scarlett_smooth_2023} (which we will not be discussing) to stratified vector bundles as well.
\end{rem}

\begin{defs}\label{whitdef}
	Let $X$ be a differentiable stratified space, $R,S\subset M$ strata so that $S\subset\overline{R}$, and let $x\in S$. Say $(R,S)$ is {\bf Whitney A regular at $x$ } (respectively, {\bf Whitney B regular at $x$}) if for any chart $\phi:U\to \bR^n$ about $x$ condition (A) (respectively, (B)) holds.
	\begin{itemize}
		\item[(A)]  If there exists a sequence $\{y_k\}\in R$ and a subspace $W\subset T_{\phi(x)} \bR^n$ such that in the $\dim(R)$ Grassmannian of $T\bR^n$,
		\bge
		\lim_{k\to\infty} T_{\phi(y_k)}\phi(U\cap R)=W,
		\ene
		then $T_{\phi(x)}\phi(U\cap S)\subset W$.
		\item[(B)] Suppose there exists sequences $\{x_k\}\subset S$ and $\{y_k\}\subset R$ with $x_k\neq y_k$ together with a subspace $W\subset T_{\phi(x)} \bR^n$ such that
		\begin{itemize}
			\item[(i)] the sequence of lines $\ell_k\subset \bR^n$ connecting $\phi(x_k)$ to $\phi(y_k)$ converges to a line $\ell$ in the projectivization of $\bR^n$,
			\item[(ii)] the sequence $T_{\phi(y_k)}\phi(U\cap R)$ converges to $W$ in the $\dim(R)$ Grassmannian of $T\bR^n$,
		\end{itemize}
		then $\ell\subset W$.
	\end{itemize}
	If all pairs of strata $(R,S)$ with $S\subset \overline{R}$ are Whitney A (respectively, Whitney B) regular at all points of $S$, then $X$ is said to be {\bf Whitney A regular} (respectively, {\bf Whitney B regular}). 
\end{defs}

\begin{rem}
	As is noted in Pfalum \cite{pflaum_analytic_2001}, Whitney B implies Whitney A. 
\end{rem}

\begin{thm}[Pflaum-Posthuma-Tang {{\cite[Corollary 5.4]{pflaum_geometry_2014}}}]\label{smthquot}
	If $G\rightrightarrows M$ is a proper Lie groupoid, then $M$ equipped with its Morita type stratification $\mathcal{S}_G(M)$ and $M/G$ with its canonical stratification $\mathcal{S}_G(M/G)$ are both canonically Whitney B stratified spaces. 
\end{thm}

\subsection{Stratified Vector Bundles}
We are now ready to give the central definition of this paper. As was stated at the beginning of this section, a stratified vector bundle can be understood to be a stratified space where the strata are the total spaces of vector bundles. 
\begin{defs}\label{svb}
	A {\bf stratified vector bundle} consists of two stratified spaces $(A,\Sigma_A)$ and $(X,\Sigma_X)$ together with a stratified morphism $p:A\to X$ which satisfy the following axioms.
	\begin{itemize}
		\item[(i)] For each $S\in\Sigma_X$, $A|_S:=p^{-1}(S)\in \Sigma_A$.
		\item[(ii)] For each $S\in\Sigma_X$, $p:A|_S\to S$ is a smooth vector bundle.
		\item[(iii)] The scalar multiplication map $\mu:\bR\times A\to A$ is a stratified morphism.
	\end{itemize}
	We will write $p:A\to X$ for a stratified vector bundle if the partitions on each space are understood from context. Call $p:A\to X$ {\bf differentiable} if $A$ and $X$ both have smooth structures, $p:A\to X$ is a smooth map, and scalar multiplication $\mu:\bR\times A\to A$ is smooth. 
\end{defs}

\begin{rem}
	Other objects called stratified vector bundles have been studied in the past. Baues and Ferrario \cite{baues_k-theory_2003} gave a notion of a stratified vector bundle making use of \enquote{$\mathbf{V}$-bundles}, which notably differs from this paper in that the total space is not demanded to be stratified. Two other authors, Kucharz and  Kurdyka \cite{kucharz_stratified-algebraic_2018} also studied objects called \enquote{algebraic-stratified vector bundles}, which thankfully are a special case of the theory we are outlining here, but in an algebraic context.
	
	As was noted earlier, contemporaneously with this paper, Scarlett defined \enquote{smooth stratified vector bundles} \cite[Definition 3.4]{scarlett_smooth_2023} in a way that is in the spirit of the definition we are using. However, Scarlett's definition is much stronger than the notion of a differentiable stratified vector bundle as in Definition \ref{svb} in that not only are $A$ and $X$ differentiable and the bundle map $p$ is smooth, but we also obtain  that the scalar multiplication $\mu$ on $A$ is locally the restriction of the scalar multiplication of a smooth vector bundle over some Euclidean space via the charts on $X$ and $A$. This definition is very well behaved and one obtains lots of useful structures like generalizations of frame bundles in this fashion. However, for the purposes of this paper, the definition is probably too strong for some of the examples we will be examining later on. 
\end{rem}

\begin{egs}\label{trivbund}
	Let $(X,\Sigma)$ be a stratified space. Consider the trivial rank $n$ vector bundle $A=X\times \bR^n$ over $X$. Letting $p=pr_1$ be the projection onto the first factor, we observe that 
	\bge
	\Sigma_A:=\{S\times \bR^n  \ | \ S\in \Sigma\}
	\ene
	is a stratification of $A$ making $p:A\to X$ into a stratified vector bundle. 
	
	Less canonically, but more generally, consider now any vector bundle $p:A\to X$ over $X$. We can also make $A$ into a stratified space. Indeed, a similar partition
	\bge
	\Sigma_A:=\{p^{-1}(S) \ | \ S\in\Sigma\}
	\ene
	almost defines a stratification of $A$, except none of the $p^{-1}(S)$ come equipped with a smooth structure. However, there is a general fact that if we are given a smooth manifold $M$ and a continuous vector bundle $\pi:E\to M$, then there exists a smooth structure on $E$ making $\pi:E\to M$ a smooth vector bundle \cite[Theorem 3.5]{hirsch_differential_1997}. In fact, this smooth structure is unique up to a smooth isomorphism. In particular, we can (non-canonically) equip each piece in $S\in\Sigma_A$ with a smooth structure making the restriction $p:p^{-1}(S)\to S$ into a smooth vector bundle. 
\end{egs}

\begin{egs}
	Suppose $(X,\Sigma)$ is a differentiable stratified space. For any chart $\phi:U\to \bR^n$, we may define
	\bge
	TU:=\bigcup_{S\in\Sigma} T(S\cap U).
	\ene
	Furthermore, since $\phi:U\cap S\to \phi(U\cap S)$ is assumed to be smooth, we can define $T\phi:TU\to T\bR^n$ by
	\bge
	T\phi|_{U\cap S}=T(\phi|_{U\cap S}).
	\ene
	Using this, define
	\bge
	TX:=\bigcup_{S\in \Sigma} TS
	\ene
	and equip $TX$ with the coarsest topology so that 
	\begin{itemize}
		\item The canonical projection $\pi:TX\to X$ is continuous.
		\item Differentials of all chart maps $T\phi:TU\to T\bR^n$ are continuous.
	\end{itemize}
	
	Observe that the natural scalar multiplication
	\bge
	\mu:\bR\times TX\to TX
	\ene
	is automatically continuous since in a chart $\phi:U\to\bR^n$, $\mu$ can be identified with the restriction of the scalar multiplication on $T\bR^n$. $TM$ also clearly has a natural partition into locally closed subsets which are topological manifolds and carry a canonical smooth structure. However, it's not necessarily the case that $\pi:TX\to X$ is a stratified vector bundle since the partition on $TX$ need not satisfy the frontier condition. This is where the Whitney A condition comes into play.
	
	\begin{thm}[Pflaum {{\cite[Theorem 2.1.2]{pflaum_analytic_2001}}}]
		If $X$ is Whitney A, then $\pi:TX\to X$ is a stratified vector bundle. Furthermore, $TX$ is a (weak) differentiable stratified space. 
	\end{thm}
	
	Thus, every Whitney A stratified space comes equipped with a canonical stratified vector bundle. Namely, its stratified tangent bundle. 
\end{egs}

\begin{defs}
	Let $p:A\to X$ and $q:B\to Y$ be two {\bf stratified vector bundles}. A morphism of stratified vector bundles consists of two stratified maps $F:X\to Y$ and $\phi:A\to B$ such that 
	\begin{itemize}
		\item[(i)] The diagram commutes
		\bge
		\begin{tikzcd}
			A\arrow[r,"\phi"]\arrow[d] & B\arrow[d]\\
			X\arrow[r,"F"] & Y
		\end{tikzcd}
		\ene
		\item[(ii)] For any strata $S\subset X$ and  $R\subset Y$ with $F(S)\subset R$, the following is a morphism of smooth vector bundles
		\bge
		\begin{tikzcd}
			A|_S\arrow[r,"\phi|_S"]\arrow[d] & B|_R\arrow[d]\\
			S\arrow[r,"F|_S"] & R
		\end{tikzcd}
		\ene
	\end{itemize}
\end{defs}

It's straightforward to show that stratified vector bundles together with morphisms of stratified vector bundles form a category. 

\begin{egs}
	Let $X$ and $Y$ be two Whitney A stratified spaces and $F:X\to Y$ a smooth map. Then, on every stratum $S\subset X$ we can find a stratum $R\subset Y$ with $F(S)\subset R$ and $F|_S:S\to R$ is a smooth map between smooth manifolds. Hence, we may define
	\bge
	T(F|_S):TS\to TR.
	\ene
	This allows us to formally define $TF:TX\to TY$ by
	\bge
	(TF)|_{TS}=T(F|_S)
	\ene
	for all strata $S\subset X$. The smoothness of $F$ together with the initial topology on $TY$ shows that $TF$ is continuous and stratum preserving. Hence, we have a stratified vector bundle morphism
	\bge
	\begin{tikzcd}
		TX\arrow[r,"TF"]\arrow[d] & TY\arrow[d]\\
		X\arrow[r,"F"] & Y
	\end{tikzcd}
	\ene
\end{egs}

\begin{defs}
	Let $p:A\to X$ be a stratified vector bundle. A {\bf section} is a continuous map $s:X\to A$ such that $p\circ s=\id_M$. Write $\Gamma(X,A)$ for the set of all sections.
\end{defs}

\begin{rem}
	It's clear that by restricting to open neighbourhoods on the base of a stratified vector bundle $p:A\to X$, we can define a sheaf:
	\bge
	\text{open }U\subset X \ \mapsto \Gamma(U,A|_U).
	\ene
\end{rem}

\subsection{Stratified Regular Monoid Actions}
One approach to studying smooth vector bundles due to Grabowski-Rotkiewicz \cite{grabowski_higher_2009} and extended by Bursztyn-Cabrera-del Hoyo \cite{bursztyn_vector_2016} to VB-groupoids and VB-algebroids is to use actions of the multiplicative monoid $(\bR,\cdot)$. In more detail, suppose we have a smooth vector bundle $\pi:E\to M$. Then, we can define a map
\bge
h:\bR\times E\to E;\quad (t,e)\mapsto t\cdot e,
\ene
where multiplication is given by scalar multiplication. This defines a smooth action of the monoid $(\bR,\cdot)$ on $E$. Writing $h_t(\cdot):=h(t,\cdot)$, this means that $h_1=\id$ and $h_t\circ h_s=h_{ts}$ for all $s,t\in \bR$.  It turns out, for suitable monoid actions, we can go the other way around and obtain a vector bundle.

Indeed, given a smooth monoid action $h:\bR\times E\to E$, write $h_t(\cdot):=h(t,\cdot)$. Then, we have 
\bge
h_0\circ h_0=h_0.
\ene
Hence, $h_0(E)$ is a closed embedded submanifold of $E$ \cite[Theorem 1.13]{kolar_natural_1993}. Therefore, we have a smooth map $h_0:E\to h_0(E)$. 

\begin{defs}\label{smthmd}
	Let $E$ be a smooth manifold and $h:\bR\times E\to E$ a smooth action by the monoid $\bR$ on $E$. Say the action is {\bf regular} if for all $e\in E$,
	\bge
	\frac{d}{dt}\bigg|_{t=0}h_t(e)=0
	\ene
	if and only if $e=h_0(e)$.
\end{defs}

\begin{thm}[Grabowski-Rotkiewicz {{\cite[Theorem 2.1]{grabowski_higher_2009}}}]\label{grab}
	The monoid action of $(\bR,\cdot)$ on the total space of a vector bundle $\pi:E\to M$ is regular. Conversely, if $h:\bR\times E\to E$ is a regular monoid action such that $h:E\to h_0(E)$ is constant rank, then $h_0:E\to h_0(E)$ is canonically a smooth vector bundle  such that the scalar multiplication is given by $h$.
\end{thm}

\begin{proof}
	The first part of the theorem can be checked by passing to local coordinates. For the latter part, define the vertical bundle
	\bge
	V_hE:=\ker( dh_0)|_{h_0(E)}.
	\ene
	Then, we get a canonical map
	\bge
	\phi:E\to V_h E;\quad e\mapsto \frac{d}{dt}\bigg|_{t=0}h_t(e).
	\ene
	The assumption that $h$ is regular guarantees that $\phi$ is a diffeomorphism. Furthermore, the map is fibre-preserving in the sense that the diagram commutes
	\bge
	\begin{tikzcd}
		E\arrow[rr,"\phi"]\arrow[dr,swap,"h_0"] && V_hE\arrow[dl,"\pi"]\\
		&h_0(E) &
	\end{tikzcd}
	\ene
	where $\pi:V_hE\to h_0(E)$ is the vector bundle projection. Hence, equip $E$ with the unique vector bundle structure making $\phi$ into an isomorphism of vector bundles. 
\end{proof}

Now with that background out of the way, let us now introduce the concept of a stratified regular monoid action. 

\begin{defs}
	Let $(A,\Sigma_A)$ be a stratified space. A monoid action $\mu:\bR\times A\to A$ is called {\bf stratified regular} if 
	\begin{itemize}
		\item[(i)] $\mu$ is a stratified morphism (see Definition \ref{stratmorph})  with respect to the stratification 
		\bge
		\{\bR\times S \ | \ S\in \Sigma_A\}
		\ene
		on $\bR\times A$.
		\item[(ii)] For each $B\in\Sigma_A$, the restriction $\mu|_B:\bR\times B\to B$ is regular in the sense of Definition \ref{smthmd}.
	\end{itemize}
	If $A$ has a smooth structure and $\mu$ is a smooth map, say the action is {\bf smooth stratified regular}. 
\end{defs}

\begin{egs}
	Let $p:A\to X$ be a (differentiable) stratified vector bundle. Then, by definition, we have a continuous scalar multiplication map
	\bge
	\mu:\bR\times A\to A.
	\ene
	The map $\mu$ is naturally stratified since the strata of $\bR\times A$ are of the form $\bR\times A|_S$, and the restriction
	\bge
	\mu|_S:\bR\times A|_S\to A|_S
	\ene
	is simply the scalar multiplication of $A|_S\to S$, which is clearly smooth. Furthermore, since over each stratum, $\mu$ is indeed the scalar multiplication of a smooth vector bundle, it follows that $\mu$ is regular.
\end{egs}

\begin{thm}\label{monthm}
	Every (smooth) stratified regular monoid action uniquely determines a (differentiable) stratified vector bundle.
\end{thm}

Before we can prove this theorem, we first need characterization of stratified projections.

\begin{lem}
	Let $(X,\Sigma_X)$ be a (differentiable) stratified space and $f:X\to X$ a (smooth) stratified map such that $f\circ f=f$ and so that $f(S)\subset S$ for all $S\in\Sigma_X$. Then, $Y:=f(X)$ together with the partition
	\bge
	\Sigma_Y:=\{S\cap Y \ | \ S\in\Sigma_X\}
	\ene
	is canonically a (differentiable) stratified space and $f:X\to Y$ is a (smooth) stratified morphism. 
\end{lem}

\begin{proof}
	First note that an easy application of the continuity of $f$ and the fact that $f=f\circ f$ automatically implies the following.
	\begin{itemize}
		\item $Y$ is a closed subspace of $X$.
		\item If $C\subset X$, then $f(\overline{C})=\overline{f(C)}$, where the closure of $f(C)$ is taken in $Y$. 
	\end{itemize}
	Furthermore, using the fact that $f$ is stratified and $f(S)\subset S$ for each stratum $S\in\Sigma_X$, it immediately follows that 
	\bge
	S=f^{-1}(f(S)).
	\ene
	
	Now, for any stratum $S\in\Sigma_X$, the subspace topology on $f(S)$ inherited from $S$ is the same as the subspace topology from $Y$. Hence, by Theorem 1.13 in \cite{kolar_natural_1993}, $f(S)$ is a topological manifold with the subspace topology from $Y$ and comes equipped with a canonical smooth manifold structure from $S$. We can also observe that $f(S)=S\cap Y$. Since $S$ is locally closed in $X$, it automatically follows that $S\cap Y$ is locally closed as well.
	
	Hence, all that we have to show is that $\Sigma_Y$ satisfies the frontier condition and we are finished showing $(Y,\Sigma_Y)$ is a stratified space. It will also automatically follow that $f:X\to Y$ is a stratified morphism.
	
	So, to show $\Sigma_Y$ satisfies the frontier condition, observe that if $R\in \Sigma_X$, then 
	\bge
	\overline{R}=\bigcup_{S}S,
	\ene
	where we are unioning over all strata $S\subset X$ such that $S\subset\overline{R}$. Hence, it follows that
	\bge
	f^{-1}(\overline{f(R)})=\overline{R}
	\ene
	Now, suppose that $S,R\in \Sigma_X$ so that $S\cap \overline{R}\cap Y\neq\emptyset$. This then implies that $S\cap \overline{R}\neq\emptyset$, hence $S\subset\overline{R}$ and so $S\cap Y\subset\overline{R\cap Y}$. 
	
	Now, suppose $X$ has a smooth structure and that $f:X\to X$ is a smooth map. Let $y\in Y$ and let $\phi:U\subset X\to \bR^n$ be a chart about $y$. It's a triviality to see that $\phi(Y\cap U)$ is locally closed in $\bR^n$, that $\phi(S\cap Y\cap U)$ is a locally closed embedded submanifold of $\bR^n$ for all $S\in\Sigma_X$, and that 
	\bge
	\phi:S\cap Y\cap U\to \phi(S\cap Y\cap U)
	\ene
	is a diffeomorphism. Hence, the restriction
	\bge
	\phi|_{Y\cap U}:Y\cap U\to \bR^n
	\ene
	is a chart. Now suppose that $\phi:U\to \bR^n$ and $\psi:V\to \bR^n$ are two charts of $X$ with $y\in U\cap V\cap Y$. Then, there exists opens $O_U,O_V\subset \bR^n$ and $y\in W\subset U\cap V$, and a diffeomorphism $H:O_U\to O_V$ such that $\phi(W)\subset O_U$, $\psi(W)\subset O_V$, and making the diagram commute
	\bge
	\begin{tikzcd}
		O_U\arrow[rr,"H"] && O_V\\
		\phi(W)\arrow[u] && \psi(W)\arrow[u]\\
		&W\arrow[ul,swap,"\phi"]\arrow[ur,"\psi"]&
	\end{tikzcd}
	\ene
	Intersecting with $Y$, we still have a commutative diagram
	\bge
	\begin{tikzcd}
		O_U\arrow[rr,"H"] && O_V\\
		\phi(W\cap Y)\arrow[u] && \psi(W\cap Y)\arrow[u]\\
		&W\cap Y\arrow[ul,swap,"\phi"]\arrow[ur,"\psi"]&
	\end{tikzcd}
	\ene
	Hence, $\phi:U\cap Y\to \bR^n$ and $\psi:V\cap Y\to \bR^n$ are compatible. Hence, the collection of charts on $Y$ given by intersecting the domains of charts on $X$ defines a smooth atlas on $Y$.
	
	Finally, to show that $f:X\to Y$ is smooth. Let $g:Y\to \bR$ be any smooth function, and fix a point $y\in Y$. Then, by definition, there exists a chart $\phi:U\to \bR^n$ of $X$ containing $y$ and a smooth function $h:\bR^n\to \bR$ so that the diagram commutes
	\bge
	\begin{tikzcd}
		U\cap Y\arrow[rr,"\phi|_{U\cap Y}"]\arrow[d,"g"] && \bR^n\arrow[dll,"h"]\\
		\bR && 
	\end{tikzcd}
	\ene
	Since $f\circ f=f$ and $y\in U\cap f(X)$, it follows that $U\cap f^{-1}(U)$ is non-empty. Thus, we may choose another chart $\psi:V\to \bR^n$ in $X$ containing $y$ so that $V\cup f(V)\subset U$ and so that $\psi=\phi|_V$. As shown by Pflaum \cite[Proposition 1.3.8]{pflaum_analytic_2001}, since $f:X\to X$ is smooth and since $f(y)=y$, there exists an open neighbourhood $W\subset V$ of $y$ and a smooth function $F:\bR^n\to \bR^n$ making the diagram commute
	\bge
	\begin{tikzcd}
		W\arrow[r,"\phi|_V"]\arrow[d,"f"] & \bR^n\arrow[d,"F"]\\
		U\arrow[r,"\phi"] & \bR^n
	\end{tikzcd}
	\ene
	It then follows that $g\circ (f|_W)=h\circ F\circ \phi|_W$. Hence, $g\circ f$ is smooth. Since $g$ was arbitrary, we conclude that $f$ is smooth.
\end{proof}

\begin{proof}(Of Theorem \ref{monthm})
	Observe that since $\mu:\bR\times A\to A$ is a stratified monoid action, it follows that $\mu_0:A\to A$ satisfies all the assumptions of the previous Lemma. Hence, $\mu_0(A)$ together with 
	\bge
	\Sigma_{\mu_0(A)}=\{S\cap \mu_0(A) \ | \ S\in\Sigma_A\}
	\ene
	is a stratified space and $\mu_0:A\to \mu_0(A)$ is a stratified map. Furthermore, $\mu_0(A)$ and $\mu_0:A\to \mu_0(A)$ are smooth if $A$ and $\mu$ are as well.
	
	So all we have to show now is that for each $S\in \Sigma_A$, the natural map $\mu_0|_S:S\to \mu_0(S)=S\cap \mu_0(A)$ is a vector bundle. However, this is clear since the restriction $\mu|_S:\bR\times S\to S$ is a smooth regular monoid action, hence is the scalar multiplication of a unique smooth vector bundle structure on $\mu_0:S\to \mu_0(S)$ by Theorem \ref{grab}.
\end{proof}

\section{Interesting Examples}
I would now like to take time to discuss certain interesting classes of examples of stratified vector bundles. We have already seen that vector bundles over stratified spaces are (non-canonically) stratified vector bundles and Whitney A stratified spaces always admit a tangent bundle. The latter are interesting since they are genuinely singular, whereas a usual vector bundle is not. We will now describe more ways in which genuinely singular stratified vector bundles can be realized.

\subsection{Singular Foliations}
For our first big class of stratified vector bundles, we discuss singular foliations. These can be understood as singular subbundles of the tangent bundle $TM$ of a smooth manifold $M$ that satisfy an integrability condition. We will first discuss the definition of singular foliations, then move on to examples where they can be stratified suitably. Let us first recall the classical Stefan-Sussmann definition of a singular foliation as found in Miyamoto \cite{miyamoto_basic_2023}. A more general discussion of generalized subbundles can be found in \cite{drager_smooth_2012}.
\begin{defs}
	Let $M$ be a smooth manifold. 
	\begin{itemize}
		\item[(i)] A {\bf smooth singular distribution} on $M$ is a subset $\Delta\subset TM$ of the tangent bundle such that for all $x\in M$, $\Delta_x:=\Delta\cap T_xM$ is a vector subspace of $x$ and so that if $v\in \Delta_x$, then there exist an open neighbourhood $U\subset M$ of $x$ and a vector field $X\in\mfX(U)$ such that $X_x=v$ and $X_y\in \Delta_y$ for all $y\in U$.
		\item[(ii)] A {\bf (Stefan-Sussmann) singular foliation} of $M$ is a partition $\mathcal{F}$ of $M$ into weakly embedded, connected submanifolds of $M$ called {\bf leaves} so that 
		\bge
		T\mathcal{F}:=\bigcup_{x\in M} T_xL_x,
		\ene
		is a smooth singular distribution, where $L_x$ is the leaf passing through $x$.
	\end{itemize}
\end{defs}

\begin{rem}
	\begin{itemize}
		\item[(1)] Another common definition of a singular foliation is the one found in in Androulidakis-Skandalis \cite{androulidakis_holonomy_2009}. Given a smooth manifold $M$, an {\bf Androulidakis-Skandalis singular foliation} is a $C^\infty(M)$ submodule $\mathscr{F}$ of $\mfX_c(M)$, the compactly supported vector fields on $M$, which is locally finitely generated and stable under Lie brackets. Every Androulidakis-Skandalis singular foliation induces a Stefan-Sussmann foliation in the following fashion. For any two $x,y\in M$, say $x\sim y$ if there exists finitely many $X_1,\dots,X_n\in \mathscr{F}$ and $t_1,\dots,t_n\in \bR$ such that
		\bge
		y=\phi^{X_1}_{t_1}\circ\cdots\circ \phi^{X_n}_{t_n}(x),
		\ene
		where $\phi^{X_i}_{t_i}$ denotes the flow of $X_i$ at time $t_i$. The resulting partition $\mathcal{F}$ into equivalence classes is then a Stefan-Sussmann singular foliation.
		
		\item[(2)] Any Androulidakis-Skandalis singular foliation induces a Stefan-Sussmann singular foliation, but this map is not injective. For example, consider the Androulidakis-Skandalis foliations on $\bR$
		\bge
		\mathscr{F}_1=C^\infty_c(\bR)\cdot \bigg\{x\frac{d}{dx}\bigg\}
		\ene
		and
		\bge
		\mathscr{F}_1=C^\infty_c(\bR)\cdot \bigg\{x^2\frac{d}{dx}\bigg\}.
		\ene
		They induce the same Stefan-Sussmann singular foliation, namely
		\bge
		\mathcal{F}=\{\bR_{<0},\bR_{>0},\{0\}\},
		\ene
		but their holonomy groupoids are not isomorphic \cite{androulidakis_holonomy_2009}.
		
		\item[(3)] For all that follows, when we say singular foliation, we mean a Stefan-Sussmann singular foliation. 
	\end{itemize}
\end{rem}

\begin{defs}\label{stratfoldef}
	Let $M$ be a smooth manifold. A {\bf stratified foliation} of $M$ consists of a singular foliation $\mathcal{F}$ of $M$ and a stratification $\Sigma$ of $M$ into embedded submanifolds such that 
	\begin{itemize}
		\item[(i)] Each stratum $S\in\Sigma$ is a union of leaves of $\mathcal{F}$.
		\item[(ii)] For each each stratum $S\in\Sigma$, if $\mathcal{F}_S$ denotes the induced foliation on $S$, then $\mathcal{F}_S$ is regular. 
	\end{itemize}
\end{defs}

\begin{egs}
	Not every singular foliation can be equipped with a compatible stratification. For example, let $C\subset[0,1]$ be the Cantor set and consider the partition $\mathcal{F}$ of $\bR$ consisting of
	\begin{itemize}
		\item the points of $C$
		\item the connected components of $\bR\setminus C$, a countable collection of disjoint open intervals.
	\end{itemize}
	This partition is indeed a foliation. Since $C$ is closed, we may find a smooth function $f:\bR\to \bR$ such that $f^{-1}(0)=C$. Define a vector field $X\in\mfX(\bR)$ by
	\bge
	X_x:=f(x)\frac{d}{dx}.
	\ene
	Clearly $X_x\in T_x\mathcal{F}$ for every $x\in \bR$. Furthermore, if $v\in T_{x_0}\mathcal{F}$ for some $x_0$, then $v=c\displaystyle\frac{d}{dx}\bigg|_{x=x_0}$. If $x_0\in C$, then $v=X_{x_0}$. Else, $v=\displaystyle\frac{c}{f(x_0)}X_{x_0}$.
	
	Suppose $\Sigma$ is a partition of $\bR$ into embedded submanifolds so that each element $S\in\Sigma$ is connected and a union of leaves. If $S\cap C\neq\emptyset$, then $S$ must be a singleton since $C$ is totally disconnected. Similarly, if $S\cap (\bR\setminus C)$, then $S$ must be a subinterval of one of the disjoint open intervals that make up $\bR\setminus C$. Hence, $\Sigma$ must be a finer partition of $\bR$ than $\mathcal{F}$. 
	
	However, this means that $\Sigma$ cannot be locally finite. Since  the Cantor set $C$ is dense and each of the points of $C$ are elements of $\Sigma$. Thus, there does not exist a stratification of $\bR$ which satisfies Definition \ref{stratfoldef}
\end{egs}

\begin{egs}
	A {\bf singular Riemannian foliation} of a manifold $M$ consists of a singular foliation $\mathcal{F}$ together with a Riemannian metric $g$ on $M$ such that if $\gamma$ is a geodesic and $L\subset M$ a leaf such that $\gamma$ is orthogonal to $L$, then $\gamma$ is orthogonal to every other leaf it intersects. 
	
	Let $x,y\in M$ and declare $x\sim y$ if and only if $\dim L_x=\dim L_y$. Now let $\Sigma$ be the partition of $M$ given by the connected components of the equivalence classes. As was shown by Molino\cite[Proposition 6.3]{molino_riemannian_1988}, $\Sigma$ is a stratification of $M$ by embedded locally closed submanifolds. 
\end{egs}

\begin{egs}
	Let $G\rightrightarrows M$ be a proper Lie groupoid and let $A\Rightarrow M$ be its Lie algebroid with anchor map $\rho:A\to TM$ \cite[Definition 13.34]{crainic_lectures_2021}. Then, $\rho(\Gamma_c(A))\subset \mfX_c(M)$ is an Androulidakis-Skandalis foliation and hence induces a Stefan-Sussmann foliation $\mathcal{F}$. It's not so hard to see that the leaves of this foliation are precisely the orbits of $G\rightrightarrows M$. 
	
	Now, let $\Sigma$ be the stratification on $M$ by Morita type. Observe that if $x,y\in M$ lie in the same orbit, then they have the same Morita type. It follows that each of the Morita type strata are unions of orbits. Proposition 4.28 and 4.29 in \cite{crainic_orbispaces_2018} imply that the induced foliation $\mathcal{F}_S$ on $S$ is regular. Hence, $(M,\Sigma,\mathcal{F})$ is a stratified foliation.
\end{egs}

\begin{thm}
	Let $(M,\mathcal{F},\Sigma)$ be a stratified foliation. Then, $T\mathcal{F}$ is a locally closed subspace of $TM$ and, together with the partition $\Sigma'=\{T\mathcal{F}_S \ | \ S\in\Sigma\}$, is canonically a differentiable stratified vector bundle over $(M,\Sigma)$.
\end{thm}

\begin{proof}
	Let $(M,\mathcal{F},\Sigma)$ be a stratified foliation and let $\pi:TM\to M$ be the projection of the projection of the usual (smooth) tangent bundle on $M$. On each piece $S\in \Sigma$, $\mathcal{F}$ induces a regular foliation $\mathcal{F}_S$. By Frobenius' Theorem, the tangent bundle $T\mathcal{F}_S$ in $TS$ is an embedded involutive subbundle. Thus, let us define
	\bge
	T\mathcal{F}:=\bigcup_{S\in\Sigma}T\mathcal{F}_S.
	\ene
	The only thing that needs to be checked is that $T\mathcal{F}$ satisfies the frontier condition. 
	
	Suppose $S,R\in\Sigma$ are two strata with $T\mathcal{F}_S\cap\overline{T\mathcal{F}_R}\neq\emptyset$, fix $v\in T\mathcal{F}_S$ and let $x=\pi(v)$. Then by continuity of $T\mathcal{F}\to M$, we have $S\cap\overline{R}\neq\emptyset$ and hence there exists a sequence $\{x_k\}\subset R$ such that $\lim _{k\to\infty}x_k=x$. Now, by assumption, there exists a vector field $X\in\mfX(M)$ such that $X_x=v$ and $X_y\in T_y\mathcal{F}$ for all $y\in M$. Hence, it immediately follows that $X_{x_k}\in T_{x_k} \mathcal{F}_R$ for all $k$ and we immediately have
	\bge
	\lim_{k\to\infty} X_{x_k}=X_x=v.
	\ene
	
	The differentiable structure on $T\mathcal{F}_R\to M$ is inherited from being a subset of a smooth manifold, namely $TM$. Both projection and scalar multiplication are restrictions of the projection and scalar multiplication on $TM$, and so both are smooth.
\end{proof}

\subsection{Equivariant Vector Bundles}
The next family of examples arises from smooth equivariant vector bundles. It is well known that if $K$ is a compact connected Lie group, $P\to M$ a principal $K$-bundle, and $\pi:E\to P$ is a $K$-equivariant vector bundle, then the induced map $E/K\to P/K\cong M$ is then a smooth vector bundle. We would now like to extend a version of this result to general proper actions of connected Lie groups. That is, given a connected Lie group $G$ acting properly on a manifold $M$ and a $G$-equivariant vector bundle $\pi:E\to M$, we will develop a procedure for getting smooth vector bundles over each of the strata of $M/G$ which will piece together into stratified vector bundles.

For all that follows, let $G$ be a connected Lie group and $M$ a proper $G$-space. Write $\mathcal{S}_G(M)$ for the orbit type stratification of $M$ (see Definition \ref{orbitypestrat}) and $\mathcal{S}_G(M/G)$ for the canonical stratification on $M/G$ (see Definition \ref{canonicalstrat}). By vector bundle, we will always mean a smooth vector bundle. To begin, let us state some easy facts about equivariant vector bundles.

\begin{prop}\label{easyprop}
	If $\pi:E\to M$ is a $G$-equivariant vector bundle and $\mu$ is its scalar multiplication, then the following holds.
	\begin{itemize}
		\item[(i)] $E$ is a proper $G$-space, hence admits an orbit type stratification $\mathcal{S}_G(E)$ and a canonical stratification on the orbit space $\mathcal{S}_G(E/G)$.
		\item[(ii)] The induced maps $\overline{\pi}:E/G\to M/G$ and $\overline{\mu}:\bR\times (E/G)\to E/G$ are smooth (not necessarily stratified) maps. 
	\end{itemize}
\end{prop}

\begin{proof}
	\begin{itemize}
		\item[(i)] Using the fact that the action of $G$ on $M$ is proper, that $\pi:E\to M$ is continuous and equivariant, and using Proposition 21.5 in \cite{lee_introduction_2013}, we immediately deduce the action on $E$ is proper.
		\item[(ii)] By the discussion in Example \ref{smoothorbits}, we have $C^\infty(E/G)\cong C^\infty (E)^G$ and $C^\infty(M/G)\cong C^\infty(M)^G$ via the quotient maps $E\to E/G$ and $M\to M/G$. Using this, one can check that $\overline{\pi}$ and $\overline{\mu}$ pull back smooth functions to smooth functions, and thus are both smooth maps of differential stratified spaces.
	\end{itemize}
\end{proof}

\begin{rem}
	The map $\overline{\pi}:E/G\to M/G$ need not be a stratified vector bundle. Indeed, consider $M=\{pt\}$, $E=\bR^2$, $\pi:\bR^2\to \{pt\}$ the only map. Let $S^1$ act on $M$ trivially and on $E$ as a subgroup of $\GL(2,\bR)$. Then, $\pi:E\to M$ is a proper $S^1$-equivariant vector bundle. However, $E/S^1$ has no natural vector space structure on it. 
	
	Indeed, let $[e]$ denote the equivalence class of $e\in\bR^2$ in $\bR^2/S^1$. Then,  given $e,e'\in\bR^2$, we could define
	\bge
	[e]+[e']:=[e+e'].
	\ene
	However, it's straightforward to show that this operation is not well-defined since for any nonzero $e\in \bR^2$, $[e]=[-e]$ and $[2e]\neq [0]$. 
	
	More seriously, $E/S^1$ is homeomorphic to $[0,\infty)$ via 
	\bge
	E/S^1;\quad [e]\mapsto \|e\|^2
	\ene
	and $[0,\infty)$ is not homeomorphic to any real finite dimensional vector space endowed with the standard topology. 
\end{rem}

Therefore, to produce a stratified vector bundle out of a $G$-equivariant vector bundle $\pi:E\to M$, we will need to modify $E$ somehow. The procedure that we outline here will involve choosing a particular subset of $E$ which is better behaved under taking quotients. This will require a local normal form that relies on a homotopic property of equivariant vector bundles.

\begin{lem}\label{homotopybund}
	Let $K$ be a compact connected Lie group, $\pi:E\to M$ be a $K$-equivariant vector bundle, and $f_0,f_1:N\to M$ two smooth $K$-equivariant maps which are $K$-homotopic to one another. Then, as $K$-equivariant vector bundles over $N$, the pullback bundles $f_0^{-1}E$ and $f_1^{-1}E$ are isomorphic.
\end{lem}

\begin{proof}
	Exactly the same proof as in Segal \cite{segal_equivariant_1968}. The only difference here is that $M$ is not assumed to be compact. However, since $M$ is paracompact, the space of sections $\Gamma(E)$ still a Fr\'{e}chet space \cite[Lemma 30.3]{kriegl_convenient_1997}, hence the averaging trick of Segal in \cite[Proposition 1.1]{segal_equivariant_1968} still holds and so do its consequences.
\end{proof}

Due to the slice theorem for proper Lie group actions, for every smooth proper $G$-space $M$ and every point $x\in M$, we can find a neighbourhood $U\subset M$ of $x$ which is $G$-equivariantly diffeomorphic to $G\times_H V$, where $H\leq G$ is a compact subgroup and $V$ is a linear $H$-representation. For all the constructions that follow, we will need to construct a local normal form for $G$-equivariant vector bundles over such a slice chart. For this purpose, suppose now that $W$ is another linear $H$-representation and define
\begin{equation}\label{slicebund}
	E=G\times_H(V\times W).
\end{equation}
We can make $E$ into a $G$-equivariant vector bundle over $G\times_H V$ by declaring
\begin{equation}
	E\to G\times_H V;\quad [g,(v,w)]\mapsto [g,v].
\end{equation}
We will now use the homotopy property of Lemma \ref{homotopybund} to show that every $G$-equivariant vector bundle locally has the form of equation (\ref{slicebund})

\begin{lem}[Local Normal Form of  an Equivariant Vector Bundle]\label{vbundnorm}
	Let $H\leq G$ be a compact subgroup and $V$ a linear $H$-representation. Then, for any $G$-equivariant vector bundle $E$ over $G\times_H V$, there exists an isomorphism of $G$-vector bundles
	\bge
	\begin{tikzcd}
		E\arrow[rr]\arrow[dr] && G\times_H(V\times W)\arrow[dl]\\
		&G\times_H V&
	\end{tikzcd}
	\ene
	where $W$ is a linear $H$-representation.
\end{lem}

\begin{proof}
	Let $\pi:E\to M$ be a $G$-equivariant vector bundle. Observe that $V$ can be identified with $H\times_H V\subset G\times_H V$. Defining $A=E|_V$, we obtain an $H$-equivariant vector bundle $\pi|_A:A\to V$ over $V$. Since $V$ is $H$-equivariantly homotopic to $\{0\}$ and $H$ is compact, it follows that there is an $H$-equivariant isomorphism of vector bundles over $V$
	\bge
	A\cong V\times W,
	\ene
	where $W$ is an $H$-representation. 
	
	Now, observe that $G\times_H A$ together with the map 
	\bge
	\widetilde{\pi}:G\times_H A\to M;\quad [g,a]\mapsto [g,\pi(a)]
	\ene
	is a $G$-equivariant vector bundle over $M$, where the fibre over a point $[g,v]\in M$ is naturally identified with $A_{[e,v]}$. It's then an easy task to check that the map
	\bge
	\phi:G\times_H A\to E;\quad [g,a]\mapsto g(a)
	\ene
	defines an isomorphism of $G$-equivariant vector bundles over $M$.
\end{proof}

Now we are ready to construct our better behaved subset of a $G$-equivariant vector bundle $\pi:E\to M$. For each $x\in M$, if we let $G_x\leq G$ denote the stabilizer subgroup of $x$, then equivariance implies that $E_x$ is a linear $G_x$-representation. We can thus define
\begin{equation}\label{tilde}
	\widetilde{E}:=\bigcup_{x\in M} E_x^{G_x}\subset E
\end{equation}
and equip $\widetilde{E}$ with the subspace topology inherited from $E$. It's clear that $\widetilde{E}$ is $G$-invariant. 

\begin{lem}\label{stratistrat}
	Let $\pi:E\to M$ be a $G$-equivariant vector bundle, and $S\in\mathcal{S}_G(M)$ an orbit type stratum of $M$. Then, $\widetilde{E}|_S$ is a smooth subbundle of $E|_S$. Furthermore, if $\tau\in\mathcal{S}_G(E)$ is an orbit type stratum of $E$ which intersects the zero section non-trivially, then $\tau=\widetilde{E}|_S$ for some orbit type stratum $S$ of $M$.
\end{lem}

\begin{proof}
	Fix a point $x\in S$ and let $H=G_x$. Working locally, we may assume
	\bge
	M=G\times_H V,
	\ene
	where $V$ is a linear $H$-representation, and we may assume that $S=M_{(H)}$. Hence,
	\bge
	S=G\times_H V^H=G/H\times V^H.
	\ene
	
	Now, by Lemma \ref{vbundnorm}, we may assume $E=G\times_H(V\times W)$, where $W$ is an $H$-representation. Hence,
	\bge
	E|_S=G\times_H(V^H\times W).
	\ene
	It's clear then that
	\bge
	\widetilde{E}|_S=G\times_H(V^H\times W^H)=G/H\times(V^H\times W^H)
	\ene
	which is clearly a smooth subbundle of $E|_S$. 
	
	Finally, let $\tau\subset E$ be an orbit type stratum of $E$ intersecting the zero section of $E$ non-trivially. Without loss of generality, we may assume that $0_x\in \tau$ for the same $x$ as above. It is then easy to see that $\tau$ is a connected component of $E_{(H)}$ and that it has the local form
	\bge
	\tau=G/H\times (V^H\times W^H),
	\ene
	which is precisely equal to $\widetilde{E}|_S$.
\end{proof}

Define two partitions, one on $\widetilde{E}$ and one on $\widetilde{E}/G$:
\begin{align}
	\mathcal{S}_G(\widetilde{E})&:=\{\widetilde{E}|_S \ | \ S\in\mathcal{S}_G(M)\}\\
	\mathcal{S}_G(\widetilde{E}/G)&:=\{\widetilde{E}|_S/G \ | \ S\in\mathcal{S}_G(M)\}.
\end{align}

\begin{thm}
	Let $\pi:E\to M$ be a $G$-equivariant vector bundle. Then, $\widetilde{E}\to M$ (see Equation (\ref{tilde})) and $\widetilde{E}/G\to M/G$ are canonically differentiable stratified vector bundles. 	
\end{thm}

\begin{proof}
	Note that by Lemma \ref{stratistrat}, it immediately follows that 
	\bge
	\mathcal{S}_G(\widetilde{E})=\{\tau\in\mathcal{S}_G(E) \ | \ \tau\text{ intersects the zero section of }E\text{ non-trivially}\}\subset \mathcal{S}_G(E)
	\ene
	and consequently we also have 
	\bge
	\mathcal{S}_G(\widetilde{E}/G)\subset \mathcal{S}_G(E/G).
	\ene
	Thus, $\mathcal{S}_G(\widetilde{E})$ and $\mathcal{S}_G(\widetilde{E}/G)$ are automatically stratifications of $\widetilde{E}$ and $\widetilde{E}/G$, respectively. Furthermore, by restricting the atlases on $E$ and $E/G$, $\widetilde{E}$ and $\widetilde{E}/G$ each inherit a (weak) differentiable stratified space structure. By Proposition \ref{easyprop}, the natural projections $\widetilde{E}\to M$ and $\widetilde{E}/G\to M/G$ are smooth and the natural monoid actions by $(\bR,\cdot)$ on both $\widetilde{E}$ and $\widetilde{E}/G$ are smooth as well. 
	
	All we have left to show is that the strata of $\widetilde{E}/G$ are smooth vector bundles over the strata of $M/G$. Choose an orbit-type stratum $S\in\mathcal{S}_G(M)$ of $M$ and a point $x\in S$ with $H=G_x$. Then, there exists $H$-representations $V$ and $W$ so that locally
	\bge
	\widetilde{E}|_S=G/H\times V^H\times W^H
	\ene
	and
	\bge
	S=G/H\times V^H,
	\ene
	with the obvious projection $\widetilde{E}|_S\to S$. Passing to the quotient, locally we have 
	\bge
	\widetilde{E}|_S/G=V^H\times W^H
	\ene
	and 
	\bge
	S/G=V^H.
	\ene
	Thus, $\widetilde{E}|_S/G\to S/G$ is a smooth vector bundle.
\end{proof}

\begin{rem}
	For a point of comparison, let $M$ be a smooth manifold and suppose that $G$ acts properly on $M$. Then, we get two natural stratified vector bundles associated to $M/G$. Indeed, first by taking the tangent bundle of $M$, we get a canonical $G$-equivariant vector bundle $\pi:TM\to M$. Using the above procedure, we get $\widetilde{TM}/G\to M/G$ as our first stratified vector bundle. On the other hand, $M/G$ is a Whitney A stratified space, and so we also have $T(M/G)$. In general, $\widetilde{TM}/G\neq T(M/G)$. Indeed, consider the example of $S^1$ acting on $\bR^2$. One can check that
	\bge
	\widetilde{T_x \bR^2}=
	\begin{cases}
		T_x \bR^2,\quad x\neq 0\\
		0,\quad x=0.
	\end{cases}
	\ene
	This implies
	\bge
	\dim(\widetilde{T \bR^2}/S^1)_{[x]}=
	\begin{cases}
		2,\quad [x]\neq [0]\\
		0,\quad [x]=[0].
	\end{cases}
	\ene
	On the other hand, for all $x\in \bR^2$
	\bge
	T_{[x]}(\bR^2/S^1)\cong T_x\bR^2/T_x(S^1\cdot x).
	\ene
	This in turn shows
	\bge
	\dim T_{[x]}(\bR^2/S^1)\cong
	\begin{cases}
		1,\quad [x]\neq [0]\\
		0,\quad [x]=[0].
	\end{cases}
	\ene
	Therefore, $(\widetilde{T \bR^2})/S^1$ and  $T(\bR^2/S^1)$ are not isomorphic as stratified vector bundles.
	
	However, these two constructions are not unrelated to one another. Indeed, let
	\begin{equation}
		A:=\bigcup_{x\in M} N_x(M,G\cdot x)
	\end{equation}
	denote the collection of all normal spaces to the orbits of the $G$-action on $M$. $A$ is not a vector bundle since the rank can change as we vary over $M$. However, for each $x\in M$, the fibre of $A$ over $x$ is a linear $G_x$ representation. Hence, we can take fibre-wise stabilizers as before. We then define a subset of $A$ closed under the $G$ action by
	\begin{equation}
		\widetilde{A}:=\bigcup_{x\in M} A_x^{G_x}.
	\end{equation}
	One can then verify that $\widetilde{A}$ together with the partition
	\begin{equation}
		\{\widetilde{A}|_S \ | \ S\in\mathcal{S}_G(M)\}
	\end{equation}
	is a stratified vector bundle over $M$. The strata are also closed under the $G$ action on $A$, hence $\widetilde{A}/G$ together with the partition
	\begin{equation}
		\{\widetilde{A}|_S/G \ | \ S\in\mathcal{S}_G(M)\}
	\end{equation}
	is also a stratified vector bundle, but now over $M/G$. By chasing diagrams, we get a canonical morphism of stratified vector bundles over $M/G$
	\begin{equation}\label{iso?}
		\widetilde{A}/G\to T(M/G).
	\end{equation}
	It is straightforward to check that this map is bijective and an isomorphism of smooth vector bundles when restricting to strata.
	
	\begin{conjecture}
		The map $\widetilde{A}/G\to T(M/G)$ in equation (\ref{iso?}) is an isomorphism of stratified vector bundles over $M/G$.
	\end{conjecture}
	
	The construction of $\widetilde{A}/G$ can be generalized to more general objects called VB-groupoids. We believe this is the proper framework for determining if the map in equation (\ref{iso?}) is an isomorphism of stratified vector bundles.
\end{rem}

\section{Linear Functors and Stratified Vector Bundles}

We now have a pretty large class of examples of stratified vector bundles. We would now like to discuss ways of getting even more examples: applying linear functors. But first, a definition.

\begin{defs}
	Let $\mathfrak{Vect}_\bR$ denote the category of finite-dimensional $\bR$-vector spaces. A {\bf linear functor} is a  functor $F:\mathfrak{Vect}_\bR\to \mathfrak{Vect}_\bR$ such that for all finite-dimensional $\bR$-vector spaces $V$ and $W$, the natural map
	\bge
	\Hom(V,W)\to \Hom(F(V),F(W))
	\ene
	is linear. 
\end{defs}

If $\pi:E\to M$ is a smooth vector bundle and $F$ is a linear functor, we can apply $F$ to $E$ fibre-wise to obtain a set-theoretic vector bundle $\pi_F:F(E)\to M$. To topologize $F(E)$ and make it into a smooth vector bundle, we make use of the local triviality of $E$. Recall, we can choose an open cover $\{U_i\}$ of $M$ together with isomorphisms of vector bundles
\bge
\phi_i:\pi^{-1}(U_i)\to U_i\times \bR^n.
\ene
Let $g_{ij}\in C^\infty(U_i\cap U_j,\text{GL}_n(\bR))$ be the resulting transition functions defined by
\bge
g_{ij}\phi_j=\phi_i
\ene
on $U_i\cap U_j$. We can then reconstruct $E$ by
\bge
E\cong \bigg(\bigsqcup_i U_i\times\bR^n\bigg)/\sim
\ene
where $(x,v)\sim (x,g_{ij}(x)v)$. The upshot of this digression is that we can now easily topologize $F(E)$. Indeed, we can apply $F$ to the transition functions to get smooth maps
\bge
F(g_{ij})\in C^\infty(U_i\cap U_j,\text{GL}(F(\bR^n)))
\ene
and declare
\bge
F(E)\cong \bigg(\bigsqcup_i U_i\times F(\bR^n)\bigg)/\sim,
\ene
where $(x,v)\sim (x,F(g_{ij}(x))v)$. 

The key to this whole discussion was the local triviality of $E$. The question is now, how can we apply a linear functor to a stratified vector bundle when it need not be locally trivial? That is, if $p:A\to X$ is a stratified vector bundle and $F$ is a linear functor, then how can we topologize 
\begin{equation}\label{functstrat}
	F(A):=\bigcup_{x\in X} F(A_x)
\end{equation}
so that (1) $F(A)$ is a stratified space and (2) the natural map $p_F:F(A)\to X$ is a stratified vector bundle?

If we were to follow the above procedure and demand that the stratified vector bundle $p:A\to X$ to be locally isomorphic to a trivial stratified vector bundle in the sense of Example \ref{trivbund}, then this would be the same thing as imposing that $p:A\to X$ be a vector bundle over a stratified space. This is obviously too strong a condition, as we would be unable to apply linear functors to objects like stratified tangent bundles. 

One remedy is to weaken what local triviality means. Instead of asking that there exists a local isomorphism to a trivial stratified vector bundle, we could ask that there only exist an embedding in a suitable sense. Many of the definitions of this section will be modelled on the construction of the stratified tangent bundle of a Whitney A stratified space.

\subsection{Injective Stratified Vector Bundles}
To begin, we develop a notion of an \enquote{injective structure}. This kind of structure on a stratified vector bundle will allow us to topologize the result of applying a linear functor fibre-wise as in Equation (\ref{functstrat}). Note that this definition resembles Definition \ref{smthdef} very closely.  

\begin{defs}
	Let $p:A\to X$ be a stratified vector bundle. 
	\begin{itemize}
		\item[(i)] A {\bf trivialization} consists of an injective stratified vector bundle morphism $\phi:A\into X\times \bR^n$ for some $n$ such that  $\phi$ is a topological embedding. 
		\item[(ii)] Suppose $\phi:p^{-1}(U)\to U\times \bR^n$ and $\psi:p^{-1}(V)\to V\times \bR^m$ are two trivializations, and let $N=\text{max}\{n,m\}$.We say they are {\bf compatible} if for any $x\in U\cap V$, there exists open $x\in W\subset U\cap V$ and an isomorphism of vector bundles $H:W\times \bR^N\to W\times \bR^N$ such that the diagram commutes
		\bge
		\begin{tikzcd}
			W\times \bR^N\arrow[rr,"H"] && W\times \bR^N\\
			W\times \bR^n\arrow[u,"\iota_n^N"] && W\times \bR^m\arrow[u,"\iota_m^N"]\\
			&p^{-1}(W)\arrow[ul,swap,"\phi"]\arrow[ur,"\psi"]&
		\end{tikzcd}
		\ene
		\item[(iii)] An {\bf atlas of local trivializations} consists of a collection of compatible local trivializations $\{\phi_i:p^{-1}(U_i)\to U_i\times \bR^{n_i}\}$ such that $\{U_i\}$ is an open cover of $X$. A maximal atlas of local trivializations will be called an {\bf injective structure} on $p:A\to X$. 
	\end{itemize}
\end{defs}

\begin{egs}
	Let $X$ be a Whitney A stratified space. Let $\mathcal{A}=\{\phi_i:U_i\to \bR^{n_i}\}$ be an atlas of $X$. Observe that by construction we get a family of local trivializations by differentiation
	\bge
	T\mathcal{A}:=\{T\phi_i:TU_i\to \phi_i(U_i)\times \bR^{n_i}\}.
	\ene
	The compatibility of the charts automatically implies compatibility of the local trivializations. Hence, $TX\to X$ is canonically an injective stratified vector bundle. 
\end{egs}

\begin{egs}
	Let $(M,\Sigma,\mathcal{F})$ be a stratified foliation and let $\mathcal{A}=\{\phi_i:U_i\to V_i\subset \bR^n\}$ be a smooth atlas of the manifold $M$ in the usual sense. Then it immediately follows that
	\bge
	T\mathcal{A}=\{T\phi_i:TU_i\to \phi_i(U_i)\times \bR^n\}
	\ene
	forms an injective structure. 
\end{egs}

Now, suppose $p:A\to X$ is an injective stratified vector bundle and $F$ a covariant linear functor. To topologize $F(A)$ as in equation (\ref{functstrat}),  consider a local trivialization $\phi:p^{-1}(U)\to U\times \bR^k$. Applying $F$ fibre-wise, we get an injective map
\begin{equation}\label{functtriv}
	F(\phi):p_{F}^{-1}(U)\to U\times F(\bR^k).
\end{equation}
Equip $F(A)$ with the coarsest topology so that the projection $p_F:F(A)\to X$ and the application of $F$ to all local trivializations $F(\phi)$ are continuous. One readily checks that $F(A)$ together with the partition
\begin{equation}\label{functpart}
	\Sigma_{F(A)}:=\{F(A)|_S \ | \ S\in \Sigma_X\},
\end{equation}
automatically satisfies most of the conditions of being an injective stratified vector bundle, except for perhaps the frontier condition. We will now need to impose further conditions on both the stratified vector bundle $p:A\to X$ and the linear functor $F$.

\subsection{Orthogonalizable Linear Functors and Grassmannians}
One key to showing $F(A)$ satisfies the frontier condition is to require that the linear functor $F$ induce a continuous map on Grassmannians in a suitable sense. As we will see, Grassmannians can be seen as subspaces of finite dimensional vector spaces by identifying a subspace with its orthogonal projection. We will now consider a class of functors which preserve this identification.

\begin{defs}
	Let $\mathfrak{Inn}_\bR$ denote the category whose objects are finite-dimensional real inner product spaces and whose morphisms are linear maps. A covariant functor $Q:\mathfrak{Inn}_\bR\to \mathfrak{Inn}_\bR$ is said to be {\bf orthogonal} if the following two conditions hold.
	\begin{itemize}
		\item[(i)]  For all inner product spaces $V$ and $W$, the natural map
		\bge
		\Hom(V,W)\to \Hom(Q(V),Q(W))
		\ene
		is linear. 
		\item[(ii)] If $V$ is an inner product space, $W\subset V$ a subspace, and $P_W:V\to V$ is the orthogonal projection onto $W$, then 
		\bge
		Q(P_W)=P_{Q(W)}.
		\ene
	\end{itemize}
	A linear functor $F$ will be said to be {\bf orthogonalizable} if there exists an orthogonal functor $Q$ such that the following diagram commutes:
	\bge
	\begin{tikzcd}
		\mathfrak{Inn}_\bR\arrow[r,"Q"]\arrow[d] & \mathfrak{Inn}_\bR\arrow[d] \\
		\mathfrak{Vect}_\bR\arrow[r,"F"] & \mathfrak{Vect}_\bR
	\end{tikzcd}
	\ene
	where each vertical arrow is the forgetful functor $\mathfrak{Inn}_\bR\to \mathfrak{Vect}_\bR$. Call $Q$ the {\bf orthogonalization} of $F$.
\end{defs}

\begin{egs}
	Pretty much all standard functors are orthogonalizable. Direct sums, tensor products, symmetric products, and wedge products are all orthogonalizable. For example, given an inner product space $(V,\bra\cdot,\cdot\ket)$, $\wedge^n V$ has the natural inner product
	\bge
	\bra v_1\wedge\cdots\wedge v_n,w_1\wedge\cdots\wedge w_n\ket:=\det(\bra v_i,w_j\ket).
	\ene
	Now suppose $W\subset V$ is a subspace and $P_W:V\to V$ is the orthogonal projection. Then I claim 
	\bge
	\wedge^nP_W=P_{\wedge^n W}.
	\ene
	We clearly see that the image of $\wedge^nP_W$ is $\wedge^n W$ and $(\wedge^nP_W)^2=\wedge^nP_W$. So, all we need to show is that $\wedge^nP_W$ is self-adjoint. For this, let $v=\wedge^n v_i$ and $u=\wedge^nu_j$ be elements of $\wedge^n V$. Then,
	\begin{align*}
		\bra \wedge^nP_W(v),u\ket=\det(\bra P_W(v_i),u_j\ket)=\det(\bra v_i,P_W(u_j)\ket)=\bra v,P_W(u)\ket.
	\end{align*} 
	Hence, $\wedge^n P_W$ is self-adjoint and  thus $\wedge^n P_W=P_{\wedge^n W}$.
\end{egs}

\begin{rem}
	Suppose $V$ is a finite dimensional real vector space and suppose $k$ is an integer such that $0\leq k\leq \dim V$. We define the the {\bf Grassmannian of $k$-dimensional subsets of $V$} by
	\begin{equation}
		G_k(V):=\{W\subset V \ | \ W\text{ linear subspace of }V\text{ and }\dim W=k\}.
	\end{equation}
	To topologize $G_k(V)$ choose an inner product on $V$. Then, for each $W\in G_k(V)$, let $P_W:V\to V$ denote the orthogonal projection onto $W$ with respect to the chosen inner product. This defines an injective map
	\bge
	G_k(V)\into \End(V);\quad W\mapsto P_W. 
	\ene
	Equip $G_k(V)$ with the subspace topology. Observe that the orthogonal group of $V$, $O(V)$, acts transitively on $G_k(V)$ via
	\bge
	O(V)\times G_k(V)\to G_k(V);\quad (T,P_W)\mapsto TP_WT^{-1}.
	\ene
	This identifies $G_k(V)$ with an orbit of a smooth action of a compact group, $O(V)$, on $\End(V)$. Hence, $G_k(V)$ is a smooth compact manifold.
	
	Note that the topology and smooth structure on $G_k(V)$ do not depend on choices. That is, different choices of inner product on $V$ and different choice of element $W\in G_k(V)$ produce equal topologies and smooth structures. 
\end{rem}

\begin{lem}
	Let $V$ be a finite-dimensional real vector space, $k\leq \dim V$, and $F:\mathfrak{Vect}_\bR\to \mathfrak{Vect}_\bR$ a covariant orthogonalizable linear functor. Write 
	\bge
	F(k):=\dim F(\bR^k).
	\ene
	Then, the natural map
	\bge
	F:G_k(V)\to G_{F(k)}(F(V));\quad W\mapsto F(W)
	\ene
	is smooth.
\end{lem}

\begin{proof}
	Choose an orthogonalization $Q$ of $F$ and choose an inner product on $V$. Then, $F(V)$ has a natural inner product structure induced by $Q$. Hence, we have inclusions
	\bge
	G_k(V)\into \End(V)
	\ene
	and
	\bge
	G_{F(k)}(F(V))\into \End(F(V)).
	\ene
	However, since $F$ is a linear functor, we do have a natural linear map
	\bge
	F:\End(V)\to \End(F(V)).
	\ene
	It is then easy to see that the diagram commutes
	\bge
	\begin{tikzcd}
		\End(V)\arrow[r] & \End(F(V))\\
		G_k(V)\arrow[u]\arrow[r] & G_{F(k)}(F(V))\arrow[u]
	\end{tikzcd}
	\ene
	since $F$ is orthogonalizable, which implies
	\bge
	F(P_W)=P_{F(W)}
	\ene
	for all $W\in G_k(V)$. The top map is linear, hence smooth, which implies the bottom map is smooth.
\end{proof}

\subsection{The Whitney A Condition}
To apply linear functors to injective stratified vector bundles, we will need to ensure that the frontier condition holds. A version of the Whitney A condition  for differentiable stratified spaces (see Definition \ref{whitdef}) can be easily adapted to the setting of injective stratified vector bundles. This can be used to then guarantee that the frontier condition survives under the image of a linear functor.

\begin{defs}
	Let $p:A\to X$ be an injective stratified vector bundle. A pair of strata $(S,R)$ of $X$ with $S\leq R$ are said to be {\bf Whitney A with respect to $p:A\to X$} if for each $x_0\in S$ and every local trivialization $\phi:p^{-1}(U)\to U\times \bR^k$ about $x_0$, if there exists
	\begin{itemize}
		\item[(i)]  a sequence $\{x_n\}\subset R$ converging to $x_0$, and
		\item[(ii)] a subspace $W\subset \{x_0\}\times \bR^n$ such that $\phi(A_{x_n})$ converges to $W$ in $X\times G_r(\bR^k)$, where $r=\text{rank}(A|_R)$,
	\end{itemize}
	then $\phi(A_{x_0})\subset W$. If every pair of strata $(S,R)$ with $S\leq R$ is Whitney A, say $p:A\to X$ is a {\bf Whitney A stratified vector bundle}.
\end{defs}

\begin{rem}
	\begin{itemize}
		\item[(i)] Observe that we have imposed no smoothness requirements on either $X$ or $A$. Let $X$ be a differentiable stratified space which is not Whitney A (for example, the modified Whitney's umbrella in Example 1.4.7 from \cite{pflaum_analytic_2001}) and $A=X\times \{0\}$ be the trivial rank $0$ vector bundle over $X$ with obvious projection $p:A\to X$. Neither $A$ nor $X$ is a Whitney A differentiable stratified space, but $p:A\to X$ {\bf is} a Whitney A stratified vector bundle. 
		\item[(ii)] As was noted earlier, it's very natural to extend the notion of Whitney A to the realm of stratified vector bundles. However, it is not quite as natural to extend Whitney B as that condition relies on the fact that we can identify $\bR^n$ with with any of its tangent spaces $T_x \bR^n$. By possibly equipping the base space $X$ with a smooth structure and $p:A\to X$ a compatible injective structure in some sense (perhaps along the lines of Scarlett's definition of a smooth stratified vector bundle \cite{scarlett_smooth_2023}), it should be possible to state what Whitney B means locally. However, it would probably be quite nontrivial to show that such a definition is coordinate independent.
	\end{itemize}
\end{rem}

\begin{egs}
	Any trivial stratified vector bundle $pr_1:X\times \bR^k\to X$ is clearly Whitney A. 
\end{egs}

\begin{egs}
	Let $X$ be a Whitney A stratified space. Clearly then $TX\to X$ is a Whitney A stratified vector bundle. 
\end{egs}

\begin{lem}
	Let $p:A\to X$ be an injective stratified vector bundle. Suppose for every $x\in X$ and $a\in A_x$, there exists neighbourhood $U\subset X$ of $x$ and a section $s:U\to A$ such that $s(x)=a$. Then $p:A\to X$ is Whitney A.
\end{lem}

\begin{proof}
	Since Whitney A is a local condition, it suffices to assume that there exists a global trivialization $A\subset X\times \bR^k$ for some $k$. Suppose $S,R\subset X$ are strata such that $S\subset \overline{R}$, $x_0\in S$, $\{x_n\}\subset R$ a sequence converging to $x_0$, and $W\subset \{x_0\}\times \bR^k$ a subspace so that
	\bge
	\lim A_{x_n}=W
	\ene
	in the $\text{rank}(A|_R)$ Grassmannian of $X\times \bR^k$. 
	
	We want to show that $A_{x_0}\subset W$. To do this, fix $(x_0,a_0)\in A_{x_0}$ and a section $s:X\to A$ satisfying $s(x_0)=(x_0,a_0)$. Write $s(x)=(x,f(x))$ for some continuous map $f:X\to \bR^k$. 
	
	Let $P_W:\bR^k\to \bR^k$ be the orthogonal projection onto $W$ and $P_{A_{x_n}}$ the orthogonal projection onto $A_{x_n}$. Then, by definition
	\bge
	\lim_{n\to \infty}\|P_{A_{x_n}}-P_W\|=0,
	\ene
	where $\|\cdot\|$ is the operator norm. However, we also have $\lim_{n\to \infty}f(x_n)=a_0$. Putting these two facts together, we obtain:
	\bge
	\lim_{n\to \infty}(f(x_n)-P_W(f(x_n)))=a_0-P_W(a_0)
	\ene
	in $\bR^k$. Since $P_{A_{x_n}}(f(x_n))=f(x_n)$, 
	\bge
	\lim_{n\to \infty}\|f(x_n)-P_W(f(x_n))\|=0,
	\ene
	we finally obtain $a_0=P_W(a_0)\in W$. 
\end{proof}

\begin{cor}
	The stratified vector bundles defined by stratified foliations $(M,\Sigma,\mathcal{F})$ are Whitney A.
\end{cor}

Finally, now that we have a host of examples, let us put everything together  regarding when we can apply a linear functor to a stratified vector bundle.

\begin{thm}\label{functonspace}
	Let $p:(A,\Sigma_A)\to (X,\Sigma_X)$ be a Whitney A stratified vector bundle and $F$ a covariant orthogonalizable linear functor. Then, $p_F:F(A)\to X$ is canonically a Whitney A stratified vector bundle.
\end{thm}

\begin{proof}
	As we saw, we can topologize $F(A)$ by demanding the natural projection $p_F:F(A)\to X$ and the application of $F$ to all local trivialization $\phi:p^{-1}(U)\to U\times \bR^k$ as in equation (\ref{functtriv}) are continuous. Equipping $F(A)$ with the partition 
	\bge
	\Sigma_{F(A)}=\{F(A)|_S \ | \ S\in \Sigma_X\},
	\ene
	we just need to verify that the frontier condition holds. Towards this end, suppose $F(A)|_S\cap \overline{F(A)|_R}\neq\emptyset$. It then follows that $S\subset\overline{R}$. Suppose now that $a_0\in F(A)|_S$ and let $p(a_0)=x_0$. We will construct $\{a_n\}\subset F(A)|_R$ so that $\lim a_n=a_0$.
	
	For the sake of convenience, let us assume that $A\subset U\times \bR^k$ so that $F(A)\subset U\times F(\bR^k)$. Suppose $\{x_n\}\subset R$ is a sequence converging to $x_0$ such that $\lim A_{x_n}=W$ exists. Then, $A_{x_0}\subset W$ by the Whitney A condition. Using the previous Lemma, we then conclude that
	\bge
	\lim F(A_{x_n})=F(W)
	\ene
	and $F(A_{x_0})\subset F(W)$. Hence, we can find a sequence $\{a_n\}\subset F(A|_R)$ with $a_n\in A_{x_n}$ such that $\lim a_n=a_0$. Hence, the frontier condition is satisfied. Clearly Whitney A is also satisfied. 
\end{proof}

As a consequence of the above theorem, let us now discuss further the issue of functoriality. Let $p:A\to X$ and $q:B\to Y$ be simply stratified vector bundles and 
\bge
\begin{tikzcd}
	A\arrow[d,"p"]\arrow[r,"\psi"] & B\arrow[d,"q"]\\
	X\arrow[r,"f"] & Y
\end{tikzcd}
\ene
a morphism. Given any linear functor $F$, define a set-theoretic map $F(\psi):F(A)\to F(B)$ fibre-wise by
\begin{equation}\label{functmorph}
	F(\psi)|_{F(A)_x}:=F(\psi|_{A_x}):F(A)_x\to F(B)_x
\end{equation}

To guarantee such a map defines once again a morphism of stratified vector bundles when $F$ orthogonalizable and $A$ and $B$ are equipped with injective structures, we will have to impose a condition on $\psi$.

\begin{defs}\label{injectivemorph}
	Let $p:A\to X$ and $q:B\to Y$ be injective stratified vector bundles and 
	\bge
	\begin{tikzcd}
		A\arrow[d,"p"]\arrow[r,"\psi"] & B\arrow[d,"q"]\\
		X\arrow[r,"f"] & Y
	\end{tikzcd}
	\ene
	a morphism of stratified vector bundles. Say $(\psi,f)$ is {\bf morphism of injective stratified vector bundles} if for any $x\in X$ the following condition holds. There exists 
	\begin{itemize}
		\item[(i)] open subsets $U\subset X$ and $V\subset Y$ with $x\in U$ and $f(U)\subset V$;
		\item[(ii)] local trivializations $\alpha:p^{-1}(U)\to U\times \bR^k$ and $\beta:p^{-1}(V)\to V\times \bR^\ell$; and 
		\item[(iii)] a morphism of vector bundles
		\bge
		\begin{tikzcd}
			U\times\bR^k\arrow[d,"pr_1"]\arrow[r,"\eta"] & V\times \bR^\ell\arrow[d,"pr_1"]\\
			U\arrow[r,"f"] & V
		\end{tikzcd}
		\ene
		such that the diagram commutes
		\bge
		\begin{tikzcd}
			U\times\bR^k\arrow[r,"\eta"] & V\times \bR^\ell\\
			p^{-1}(U)\arrow[u,"\alpha"]\arrow[r,"\psi"] & q^{-1}(U)\arrow[u,"\beta"]
		\end{tikzcd}
		\ene
	\end{itemize}
\end{defs}

\begin{rem}
	It's clear that morphisms of injective stratified vector bundles are closed under composition.
\end{rem}

\begin{egs}
	If $X$ and $Y$ are Whitney A stratified spaces and $f:X\to Y$ is a smooth map, then $Tf:TX\to TY$ is a morphism of injective stratified vector bundles. 
\end{egs}

\begin{lem}\label{functonmorph}
	Suppose $p:A\to X$ and $q:B\to Y$ are Whitney A stratified vector bundles, 
	\bge
	\begin{tikzcd}
		A\arrow[d,"p"]\arrow[r,"\psi"] & B\arrow[d,"q"]\\
		X\arrow[r,"f"] & Y
	\end{tikzcd}
	\ene
	is a morphism of injective stratified vector bundles, and $F$ is an orthogonalizable linear functor. Then $F(\psi):F(A)\to F(B)$ as in Equation (\ref{functmorph}) is also a morphism of injective stratified vector bundles.
\end{lem}

\begin{proof}
	Since $\psi$ is a morphism of injective stratified vector bundles, around any $x\in X$, we may choose open subsets $U\subset X$, $V\subset Y$, trivializations $\alpha:p^{-1}(U)\to U\times \bR^k$, $\beta:q^{-1}(V)\to V\times \bR^\ell$, and vector bundle morphism $\eta:U\times \bR^k\to V\times \bR^\ell$ over $f$ so that the diagram commutes
	\begin{equation}\label{morphextend}
		\begin{tikzcd}
			U\times\bR^k\arrow[r,"\eta"] & V\times \bR^\ell\\
			p^{-1}(U)\arrow[u,"\alpha"]\arrow[r,"\psi"] & q^{-1}(U)\arrow[u,"\beta"]
		\end{tikzcd}
	\end{equation}
	Since $\eta$ is an actual vector bundle morphism follows that $\eta$ has the form $\eta=(f,H)$, where $H:U\to \Hom(\bR^k,\bR^\ell)$ is a continuous map. Define $F(H)$ by the composition
	\bge
	\begin{tikzcd}
		U\arrow[rr,"H"]\arrow[ddrr,swap,"F(H)"] && \Hom(\bR^k,\bR^\ell)\arrow[dd,"F"]\\
		&&\\
		&& \Hom(F(\bR^k),F(\bR^\ell))
	\end{tikzcd}
	\ene
	Clearly $F(H)$ is continuous and thus we get a vector bundle morphism 
	\bge
	F(\eta):=(f,\eta(H)):U\times F(\bR^k)\to V\times F(\bR^\ell)
	\ene
	over $f$. Applying $F$ to the diagram in Equation (\ref{morphextend}) we get a commutative diagram of maps between sets:
	\bge
	\begin{tikzcd}
		U\times F(\bR^k)\arrow[r,"F(\eta)"] & V\times F(\bR^\ell)\\
		p_F^{-1}(U)\arrow[u,"F(\alpha)"]\arrow[r,"F(\psi)"] & q_F^{-1}(U)\arrow[u,"F(\beta)"]
	\end{tikzcd}
	\ene
	By the definition of the topologies on $F(A)$ and $F(B)$, it follows that $F(\psi)$ is continuous and hence a morphism of injective stratified vector bundles.
\end{proof}

Let $\mathfrak{Asvb}$ denote the category whose objects are Whitney A stratified vector bundles and whose morphisms are morphisms of injective stratified vector bundles. As our final result, we get the following easy consequence of Theorem \ref{functonspace} and Lemma \ref{functonmorph}.

\begin{cor}
	Any orthogonalizable linear functor $F$ defines a functor $\mathfrak{Asvb}\to \mathfrak{Asvb}$.
\end{cor}

\bibliographystyle{amsalpha}
\bibliography{SVB}

\providecommand{\bysame}{\leavevmode\hbox to3em{\hrulefill}\thinspace}
\providecommand{\MR}{\relax\ifhmode\unskip\space\fi MR }
\providecommand{\MRhref}[2]{%
  \href{http://www.ams.org/mathscinet-getitem?mr=#1}{#2}
}
\providecommand{\href}[2]{#2}
\begin{thebibliography}{NGSDS03}

\bibitem[AFT17]{ayala_local_2017}
David Ayala, John Francis, and Hiro~Lee Tanaka, \emph{Local structures on
  stratified spaces}, Advances in Mathematics \textbf{307} (2017), 903--1028
  (en).

\bibitem[AS09]{androulidakis_holonomy_2009}
Iakovos Androulidakis and Georges Skandalis, \emph{The holonomy groupoid of a
  singular foliation}, Journal für die Reine und Angewandte Mathematik.
  \textbf{626} (2009), 1--37. \MR{2492988}

\bibitem[BCdH16]{bursztyn_vector_2016}
Henrique Bursztyn, Alejandro Cabrera, and Matias del Hoyo, \emph{Vector bundles
  over {Lie} groupoids and algebroids}, Advances in Mathematics \textbf{290}
  (2016), 163--207, arXiv:1410.5135 [math].

\bibitem[BF03]{baues_k-theory_2003}
Hans-Joachim Baues and Davide~L. Ferrario, \emph{K-theory of stratified vector
  bundles}, K-Theory. An Interdisciplinary Journal for the Development,
  Application, and Influence of K-Theory in the Mathematical Sciences
  \textbf{28} (2003), no.~3, 259--284. \MR{2017530}

\bibitem[CFM21]{crainic_lectures_2021}
Marius Crainic, Rui~Loja Fernandes, and Ioan Mărcut, \emph{Lectures on
  {Poisson} geometry}, Graduate studies in mathematics, no. 217, American
  Mathematical Society, Providence, Rhode Island, 2021.

\bibitem[CM18]{crainic_orbispaces_2018}
Marius Crainic and João~Nuno Mestre, \emph{Orbispaces as differentiable
  stratified spaces}, Lett Math Phys \textbf{108} (2018), no.~3, 805--859 (en),
  arXiv:1705.00466 [math].

\bibitem[CS13]{crainic_linearization_2013}
Marius Crainic and Ivan Struchiner, \emph{On the linearization theorem for
  proper {Lie} groupoids}, Annales scientifiques de l'École Normale
  Supérieure \textbf{46} (2013), no.~5, 723--746 (en).

\bibitem[DLPR12]{drager_smooth_2012}
Lance~D. Drager, Jeffrey~M. Lee, Efton Park, and Ken Richardson, \emph{Smooth
  distributions are finitely generated}, Ann Glob Anal Geom \textbf{41} (2012),
  no.~3, 357--369 (en).

\bibitem[FPS23]{farsi_differentiable_2023}
Carla Farsi, Markus~J. Pflaum, and Christopher Seaton, \emph{Differentiable
  stratified groupoids and a de {Rham} theorem for inertia spaces}, Journal of
  Geometry and Physics \textbf{187} (2023), 104806.

\bibitem[GM80]{goresky_intersection_1980}
Mark Goresky and Robert MacPherson, \emph{Intersection homology theory},
  Topology \textbf{19} (1980), no.~2, 135--162 (en).

\bibitem[GR09]{grabowski_higher_2009}
Janusz Grabowski and Mikołaj Rotkiewicz, \emph{Higher vector bundles and
  multi-graded symplectic manifolds}, Journal of Geometry and Physics
  \textbf{59} (2009), no.~9, 1285--1305. \MR{2541820}

\bibitem[Hal13]{hall_quantum_2013}
Brian~C. Hall, \emph{Quantum {Theory} for {Mathematicians}}, Graduate {Texts}
  in {Mathematics}, vol. 267, Springer, New York, NY, 2013 (en).

\bibitem[Ham10]{hamilton_locally_2010}
Mark Hamilton, \emph{Locally toric manifolds and singular {Bohr}-{Sommerfeld}
  leaves}, Memoirs of the {American} {Mathematical} {Society}, vol. 207,
  American Mathematical Society, September 2010 (en), ISSN: 0065-9266,
  1947-6221 Issue: 971.

\bibitem[Hir97]{hirsch_differential_1997}
Morris~W. Hirsch, \emph{Differential topology}, corr. 6th print ed., Graduate
  texts in mathematics, no.~33, Springer, New York, 1997.

\bibitem[KK18]{kucharz_stratified-algebraic_2018}
Wojciech Kucharz and Krzysztof Kurdyka, \emph{Stratified-algebraic vector
  bundles}, Journal für die reine und angewandte Mathematik \textbf{2018}
  (2018), no.~745, 105--154 (en), arXiv:1308.4376 [math].

\bibitem[KM97]{kriegl_convenient_1997}
Andreas Kriegl and Peter~W. Michor, \emph{The convenient setting of global
  analysis}, Mathematical {Surveys} and {Monographs}, vol.~53, American
  Mathematical Society, Providence, RI, 1997. \MR{1471480}

\bibitem[KMS93]{kolar_natural_1993}
Ivan Kolář, Peter~W. Michor, and Jan Slovák, \emph{Natural operations in
  differential geometry}, Springer-Verlag, Berlin, 1993. \MR{1202431}

\bibitem[Lee13]{lee_introduction_2013}
John~M. Lee, \emph{Introduction to smooth manifolds}, 2nd ed ed., Graduate
  texts in mathematics, no. 218, Springer, New York ; London, 2013, OCLC:
  ocn800646950.

\bibitem[Mat73]{mather_stratifications_1973}
John~N. Mather, \emph{Stratifications and {Mappings}}, Dynamical {Systems}
  (M.~M. Peixoto, ed.), Academic Press, January 1973, pp.~195--232.

\bibitem[Mat77]{mather_differentiable_1977}
\bysame, \emph{Differentiable invariants}, Topology \textbf{16} (1977), no.~2,
  145--155.

\bibitem[Mat12]{mather_notes_2012}
John Mather, \emph{Notes on {Topological} {Stability}}, Bull. Amer. Math. Soc.
  \textbf{49} (2012), no.~4, 475--506 (en).

\bibitem[Miy23]{miyamoto_basic_2023}
David Miyamoto, \emph{The {Basic} de {Rham} {Complex} of a {Singular}
  {Foliation}}, International Mathematics Research Notices \textbf{2023}
  (2023), no.~8, 6364--6401.

\bibitem[Mol88]{molino_riemannian_1988}
Pierre Molino, \emph{Riemannian foliations}, Progress in {Mathematics},
  vol.~73, Birkhäuser Boston, Inc., Boston, MA, 1988. \MR{932463}

\bibitem[NGSDS03]{navarro_gonzalez_c-differentiable_2003}
Juan~A. Navarro~González and Juan~B. Sancho De~Salas,
  \emph{C$^\infty$-{Differentiable} {Spaces}}, Lecture {Notes} in
  {Mathematics}, vol. 1824, Springer, Berlin, Heidelberg, 2003.

\bibitem[NV23]{nocera_whitney_2023}
Guglielmo Nocera and Marco Volpe, \emph{Whitney stratifications are conically
  smooth}, Sel. Math. New Ser. \textbf{29} (2023), no.~5, 68 (en).

\bibitem[Pfl01]{pflaum_analytic_2001}
Markus~J. Pflaum, \emph{Analytic and {Geometric} {Study} of {Stratified}
  {Spaces}}, Lecture {Notes} in {Mathematics}, vol. 1768, Springer, Berlin,
  Heidelberg, 2001 (en).

\bibitem[PPT14]{pflaum_geometry_2014}
Markus~J. Pflaum, Hessel Posthuma, and Xiang Tang, \emph{Geometry of orbit
  spaces of proper {Lie} groupoids}, Journal für die reine und angewandte
  Mathematik (Crelles Journal) \textbf{2014} (2014), no.~694, 49--84 (en),
  Publisher: De Gruyter Section: Journal für die reine und angewandte
  Mathematik.

\bibitem[RS17]{resende_open_2017}
Pedro Resende and João~Paulo Santos, \emph{Open quotients of trivial vector
  bundles}, Topology and its Applications \textbf{224} (2017), 19--47 (en).

\bibitem[Sca23]{scarlett_smooth_2023}
Varun~Kher Scarlett, \emph{Smooth {Stratified} {Vector} {Bundles} and
  {Obstructions} to {Their} {Orthonormal} {Frame} {Bundles}}, Thesis, Virginia
  Tech, May 2023, Accepted: 2023-05-24T08:00:30Z Artwork Medium: ETD Interview
  Medium: ETD.

\bibitem[Sch75]{schwarz_smooth_1975}
Gerald~W. Schwarz, \emph{Smooth functions invariant under the action of a
  compact {Lie} group}, Topology \textbf{14} (1975), no.~1, 63--68 (en).

\bibitem[Seg68]{segal_equivariant_1968}
Graeme Segal, \emph{Equivariant {K}-theory}, Institut des Hautes Études
  Scientifiques. Publications Mathématiques (1968), no.~34, 129--151.
  \MR{234452}

\bibitem[SL91]{sjamaar_stratified_1991}
Reyer Sjamaar and Eugene Lerman, \emph{Stratified {Symplectic} {Spaces} and
  {Reduction}}, Annals of Mathematics \textbf{134} (1991), no.~2, 375--422,
  Publisher: Annals of Mathematics.

\bibitem[Tho69]{thom_ensembles_1969}
R.~Thom, \emph{Ensembles et morphismes stratifiés}, Bull. Amer. Math. Soc.
  \textbf{75} (1969), no.~2, 240--284 (en).

\bibitem[Vis05]{vistoli_grothendieck_2005}
Angelo Vistoli, \emph{Grothendieck topologies, fibered categories and descent
  theory.}, Fundamental algebraic geometry: {Grothendieck}'s {FGA} explained,
  Mathematical surveys and monographs, no. v. 123, American Mathematical
  Society, Providence, R.I, 2005, OCLC: ocm61362228.

\bibitem[Wey16]{weyl_classical_2016}
Hermann Weyl, \emph{The {Classical} {Groups}: {Their} {Invariants} and
  {Representations}}, The {Classical} {Groups}, Princeton University Press,
  June 2016 (en).

\bibitem[Whi15]{whitney_local_2015}
Hassler Whitney, \emph{Local {Properties} of {Analytic} {Varieties}}, Local
  {Properties} of {Analytic} {Varieties}, Princeton University Press, December
  2015, pp.~205--244 (en).

\end{thebibliography}

\end{document}